\documentclass[11pt,a4paper]{article}
\usepackage[hmargin=1.25in, vmargin=1.25in, marginparwidth=1in, marginparsep=0.1in, a4paper, centering]{geometry}
\usepackage[T1]{fontenc}
\usepackage[lining]{ebgaramond}
\usepackage{amsmath,amsthm}
\usepackage[ebgaramond]{newtxmath}
\usepackage{mathtools,apptools}
\usepackage{bm,enumitem,microtype}
\usepackage[font=small, labelfont=bf, labelsep=colon, justification=centerlast, margin=0.5in]{caption}
\usepackage{verbatim}
\usepackage{array,booktabs,longtable}
\usepackage{graphicx}
\usepackage[dvipsnames,svgnames]{xcolor}%
\usepackage[colorlinks,allcolors=blue,pagebackref]{hyperref}
\usepackage{pifont}
\renewcommand*{\backrefalt}[4]{%
\ifcase #1 %
\textcolor{red}{No citations}%
\or
\ding{43}~p.~#2%
\else
\ding{43}~pp.~#2%
\fi}
\usepackage{breakcites}
\numberwithin{equation}{section}
\theoremstyle{plain}
\newtheorem{theorem}{Theorem}[section]
\newtheorem{proposition}[theorem]{Proposition}
\newtheorem{conjecture}[theorem]{Open Problem}
\newtheorem{lemma}[theorem]{Lemma}
\newtheorem{corollary}[theorem]{Corollary}
\theoremstyle{definition}
\newtheorem{remark}[theorem]{Remark}
\newtheorem{example}[theorem]{Example}

\newcommand{\R}{\mathbb{R}}
\newcommand{\dd}{\,\mathrm{d}}
\newcommand{\D}{\mathbb{D}}
\newcommand{\jone}{j_{1,1}}
\newcommand{\fhat}[1]{\widehat{\chi_{#1}}}
\newcommand{\ev}[1]{\mathbf{e}_{#1}}
\newcommand{\N}{\mathcal{N}}

\newcommand{\dr}{\mathrm{d}}
\newcommand{\er}{\mathrm{e}}
\newcommand{\ir}{\mathrm{i}}
\newcommand{\sgn}{\operatorname{sign}}

\newcommand{\nj}[1]{\mathbf{n}_{#1}}

\newcommand{\mydoi}[1]{\href{https://doi.org/#1}{doi: #1}}
\newcommand{\myarXiv}[1]{\href{https://arxiv.org/abs/#1}{arXiv: #1}}
\newcommand{\xb}{\mathbf{x}}
\newcommand{\yb}{\mathbf{y}}
\newcommand{\xib}{\boldsymbol{\xi}}
\newcommand{\etab}{\boldsymbol{\eta}}
\newcommand{\zb}{\mathbf{z}}
\usepackage{mleftright}
\mleftright
\usepackage{marginfix}
\usepackage[backgroundcolor=white]{todonotes}
\renewcommand\footnotemark{}
\usepackage[nobottomtitles,pagestyles]{titlesec}
\titleformat{\section}
{\normalfont\large\bfseries}
{\filcenter\IfAppendix{Appendix }{\S}\thesection.}{1ex}{\filcenter}
\titleformat{\subsection}
{\normalfont\bfseries}
{\filcenter\S\thesubsection.}{1ex}{\filcenter}
\title{An isoperimetric problem for Fourier zeros of centrally symmetric convex bodies%
}
\author{
Bogdan Georgiev
\thanks{%
\textbf{B. G.:} Google DeepMind, Handyside Street, Kings Cross, London N1C 4UZ, UK;
\href{mailto:bogeorgiev@google.com}{\nolinkurl{bogeorgiev@google.com}}; \url{https://www.bogdan-georgiev.com}
}
\and
Javier G\'omez-Serrano
\thanks{%
\textbf{J. G.-S.:} Department of Mathematics,
Brown University,
314 Kassar House, 151 Thayer St.
Providence, RI 02912, USA; \href{mailto:javier_gomez_serrano@brown.edu}{\nolinkurl{javier_gomez_serrano@brown.edu}}; \url{https://sites.brown.edu/jgs}; ORCID: 0000-0002-5962-0859%
}
\and
Michael Levitin
\thanks{%
\textbf{M. L.: }Department of Mathematics and Statistics, University of Reading, 
Pepper Lane, Whiteknights, Reading RG6 6AX, UK;
\href{mailto:M.Levitin@reading.ac.uk}{\nolinkurl{M.Levitin@reading.ac.uk}}; \url{https://www.michaellevitin.net};  ORCID: 0000-0003-0020-3265%
}
\and
Daniel Platt
\thanks{%
\textbf{D. P.:} Department of Mathematics, Imperial College London, 180 Queen's Gate, South Kensington, London SW7 2RH, UK;
\href{mailto:d.platt@imperial.ac.uk}{\nolinkurl{d.platt@imperial.ac.uk}}; \url{https://dplatt.de}; ORCID: 0000-0001-9805-6579%
}
\and Iosif Polterovich
\thanks{%
\textbf{I. P.: }D\'e\-par\-te\-ment de math\'ematiques et de statistique, Univer\-sit\'e de Mont\-r\'eal, 
CP 6128 succ Centre-Ville, Mont\-r\'eal QC  H3C 3J7, Canada;
\href{mailto:iosif.polterovich@umontreal.ca}{\nolinkurl{iosif.polterovich@umontreal.ca}}; \url{http://www.dms.umontreal.ca/\~iossif}; ORCID: 0009-0007-0052-6589%
}
}
\date{\small 29 November 2026}

\begin{document}
\maketitle

\begin{abstract}
Given a centrally symmetric convex body, consider the zero set of the Fourier transform of its characteristic function. We study how large the distance from this set to the origin can be when the volume of the body is fixed. A 2009 conjecture of Benguria, Levitin, and Parnovski asserts that the maximum is attained by a Euclidean ball. We disprove it in all dimensions greater than one. In the plane, every regular centrally symmetric polygon with at least twelve sides outperforms the disk, albeit by a tiny margin, and the regular dodecagon is best among them. In dimensions three and higher, no maximiser exists.
\end{abstract}

{\small \tableofcontents}

\section{Introduction and main results}\label{sec:intro}

\subsection{The functional $\kappa$}

Let $\Omega\subset\R^d$ be a bounded convex  domain which is centrally symmetric with respect to the origin. In what follows, we refer to domains satisfying the latter property as {\em balanced}.  Let $\chi_\Omega$ denote the characteristic function of $\Omega$. Consider the zero set of its Fourier transform,
\[
\fhat{\Omega}(\xib):=\int_\Omega \er^{\ir\xib\cdot\xb}\dd \xb,
\qquad
\mathcal{N}(\Omega):=\{\xib\in\R^d:\fhat{\Omega}(\xib)=0\}.
\]
Note that since $\Omega$ is centrally symmetric, $\fhat{\Omega}$ is a real-valued function on $\R^d$.  The following basic result holds.

\begin{proposition}\label{prop:existence}
Let $\Omega\subset\R^d$ be a balanced set of finite positive Lebesgue measure.
Then $\mathcal{N}(\Omega)\ne\varnothing$. 
\end{proposition}

Note that this proposition does not assume $\Omega$ to be convex. For the proof, see \S\ref{ssec:ideas}.

The set $\mathcal{N}(\Omega)$  has been actively studied in the context of 
the Pompeiu problem and Schiffer's conjecture (see \cite{CS26, CLDP26} for most recent developments) as well as in relation to spectral estimates, see \cite{BLP09} and references therein. From the viewpoint of spectral geometry, of particular interest is the quantity
\[
\kappa(\Omega):=\operatorname{dist}(\mathcal{N}(\Omega),0).
\]
It is the distance from the origin to the real null variety of the Fourier transform of $\chi_\Omega$, which is finite by Proposition~\ref{prop:existence}. 

Let $\Omega^{\ast}$ be the Euclidean ball of the same volume as $\Omega$. 
It was  conjectured in \cite[Conjecture 2.2]{BLP09} that for every convex balanced $\Omega\subset\R^d$,
\begin{equation}\label{eq:BLP}
\kappa(\Omega)\le \kappa(\Omega^{\ast}),
\end{equation}
with equality only when $\Omega$ is a ball. 

With $s>0$, the substitution $\xb\mapsto s\xb$ gives
\begin{equation}\label{eq:scaleFT}
\fhat{s\Omega}(\xib)=s^{d}\,\fhat{\Omega}(s\xib),
\end{equation}
so $\mathcal N(s\Omega)=s^{-1}\mathcal N(\Omega)$ and
$\kappa(s\Omega)=s^{-1}\kappa(\Omega)$; since $(s\Omega)^{\ast}=s\Omega^{\ast}$,
inequality \eqref{eq:BLP} is scale-invariant.

Some evidence for this conjecture has been provided in two dimensions,
namely, its weaker version was proved in \cite[Theorem 2.4]{BLP09} with an additional multiplicative constant $2j_{0,1}/\jone\approx 1.2552$  in the right-hand side. Also, it was proved for star-shaped $C^2$ perturbations of the disk \cite[Theorem 2.7]{BLP09}. Inequality \eqref{eq:BLP} was also verified for cuboids in arbitrary dimension.

The goal of this paper is to  show that the conjectured inequality \eqref{eq:BLP} is false in all dimensions $d\ge 2$.
\subsection{Unboundedness of $\kappa$ in higher dimensions}
Let us start with the easier case $d\ge 3$. For any $\alpha>0$, consider a bipyramid, see Figure~\ref{fig:bipyramid},
\begin{equation}\label{eq:bipyramid}
\Omega_{\alpha,d}:=\left\{(x,\yb)\in\R\times\R^{d-1}:\,|x|+\alpha|\yb|<1\right\}.
\end{equation}

\begin{figure}
\centering
\includegraphics[width=0.9\textwidth]{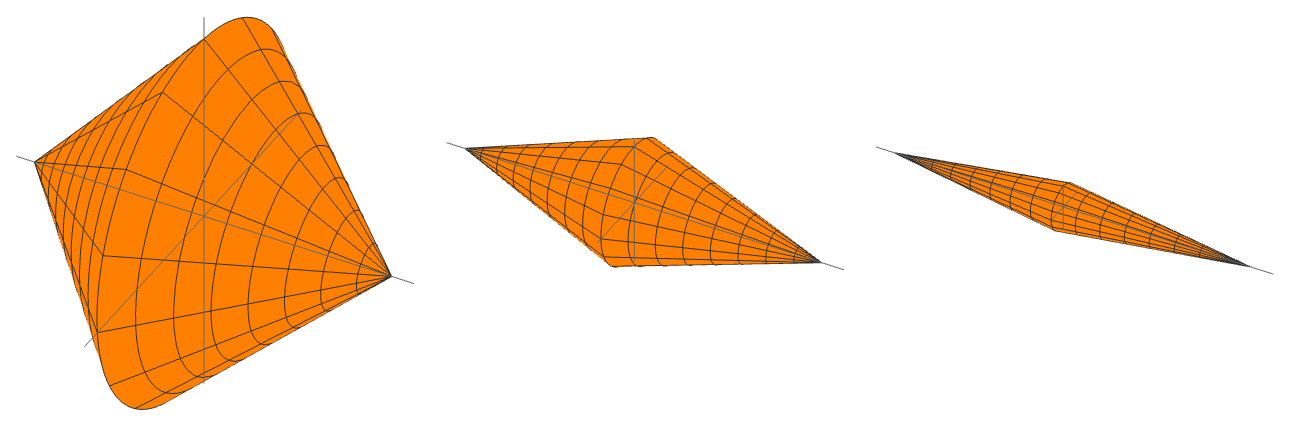}
\caption{The set $\Omega_{\alpha,d}$ for $d=3$ and varying values of $\alpha$, increasing from left to right.}
\label{fig:bipyramid}
\end{figure}

It is immediate that  $\Omega_{\alpha,d} \subset \R^d$ is a bounded convex balanced domain.

\begin{theorem}[Counterexamples for $d\ge 3$]\label{thm:highd}
There exists $\alpha_d^\ast>0$ such that
\begin{equation}\label{eq:highd_main}
\kappa\left(\Omega_{\alpha,d}\right)>\kappa\left(\Omega_{\alpha,d}^{\ast}\right)\qquad\text{for all }\alpha>\alpha_d^\ast.
\end{equation}
Moreover,  there exists a constant $C_d>0$ such that
\begin{equation}\label{eq:highd_rate}
\liminf_{\alpha\to +\infty} \frac{\kappa\left(\Omega_{\alpha,d}\right)}{\kappa\left(\Omega_{\alpha,d}^{\ast}\right) \alpha^{1/d}}\ge C_d.
\end{equation}
\end{theorem}

As an immediate consequence of \eqref{eq:highd_rate} we obtain
\begin{corollary}\label{cor:diverge}
In dimensions $d \ge 3$, 
\[
\sup_{\substack{\Omega \subset \mathbb{R}^d \\ |\Omega|=V}} \kappa(\Omega)=+\infty, 
\]
where the supremum is taken over all bounded convex balanced sets of fixed volume $V>0$.
\end{corollary}

\begin{remark} 
The construction of the domains $\Omega_{\alpha,d}$ giving a counterexample to \eqref{eq:BLP}  is  inspired by \cite[Remark 4.11]{BLP09}. We note that this construction only works for $d\ge 3$, see Remark~\ref{rem:worksd3}.
\end{remark}

\subsection{Two dimensions: regular polygons vs the disk}\label{ssec:planar-vs-disk}
Consider now the case $d=2$. It follows from the results of \cite{BLP09}
that 
\[
\sup_{\substack{\Omega \subset \mathbb{R}^2 \\ |\Omega|=A}} \kappa(\Omega) < +\infty,
\]
where the supremum is taken over all bounded convex balanced sets of given area $A>0$. Moreover, as was shown in \cite{RyadovkinFilonov2014}, there exists a domain attaining the supremum. 

Denote by $P_{2n}$, $n\ge 2$,  a regular $2n$-gon of area $\pi$ centred at the origin. Also, let $\mathbb{D}$ be the unit disk centred at the origin. As was shown in \cite{BLP09},  $\kappa(\D)=\jone\approx 3.831705970208$, where $\jone$ is the first positive zero of the Bessel function of the first kind $J_{1}$.

Numerical computation gives the following approximate values of $\kappa(P_{2n})-\jone$:
\[
\begin{array}{c|c}
\toprule
n & \kappa(P_{2n})-\jone\\\midrule
2 & -2.867983\times 10^{-1}\\
3 & -2.245474\times 10^{-2}\\
4 & -8.667777\times 10^{-4}\\
5 & -5.399272\times 10^{-6}\\
6 & +6.273756\times 10^{-6}\\
7 & +2.644535\times 10^{-6}\\
8 & +1.177537\times 10^{-6}\\
9 & +5.766065\times 10^{-7}\\
10 & +3.048199\times 10^{-7}\\
\bottomrule
\end{array}
\]
We note that the difference between the values of $\kappa(P_{2n})$ and 
$\kappa(\D)$ is very small, and without a high-precision numerical calculation it would not have been possible to detect that $\kappa(P_{2n}) > \kappa(\D)$ for any $n \ge 6$. Our next  theorem makes this observation rigorous.

\begin{theorem}\label{thm:monot}
The following inequalities hold for any $n\ge 6$:
\begin{equation}\label{eq:monotonicity}
\kappa(P_{2n}) > \kappa(P_{2n+2}) > \kappa(\D) > \kappa(P_{10}) >\kappa(P_8) > \kappa(P_6)> \kappa(P_4).
\end{equation}
\end{theorem}

Theorem~\ref{thm:monot} suggests a new candidate for the maximising domain for the functional $\kappa$.

\begin{conjecture}\label{conj:dodecagon}
Does the regular dodecagon $P_{12}$ maximise the functional $\kappa$ among all convex balanced planar domains of area $\pi$?
\end{conjecture}

\subsection{Ideas of the proofs}\label{ssec:ideas}

We start with a short
\begin{proof}[Proof of Proposition~\ref{prop:existence}]
Set $f:=\chi_\Omega$ and suppose, to the contrary, that $\mathcal N(\Omega)=\varnothing$.
Then, since  $\widehat{\chi_\Omega}(0)=|\Omega|>0$, it follows that $\widehat f(\xib) > 0$ for every $\xib\in\R^d$. Let
\[
G_t(x):=(4\pi t)^{-d/2}e^{-|x|^2/(4t)}, \qquad t>0.
\]
Since $\left|\widehat f(\xib)\right|\le |\Omega|$ for all $\xib\in\mathbb{R}^d$, $\widehat f(\xib)e^{-t|\xib|^2}\in L^1(\R^d)$, and Fourier inversion gives
\[
(2\pi)^{-d}\int_{\R^d}\widehat f(\xib)e^{-t|\xib|^2}\,\dr\xib
=(f*G_t)(0)\leq1.
\]
As the integrand is nonnegative, monotone convergence as $t\downarrow0$ yields
$\widehat f\in L^1(\R^d)$. The Fourier inversion theorem therefore implies that $f$ agrees almost everywhere with a continuous function  on $\R^d$,  which is impossible since $f$ is a characteristic function of a set of finite positive Lebesgue measure.
\end{proof}

The proof of Theorem~\ref{thm:highd} exploits the strong anisotropy of the bipyramids $\Omega_{\alpha,d}$ defined by \eqref{eq:bipyramid}. Writing
\[
\xib=(u,\etab)\in \mathbb R\times\mathbb R^{d-1},
\]
and slicing perpendicular to the distinguished $x$-axis, one obtains the Fourier transform representation \eqref{eq:fhat_alphad}, namely
\[
\widehat{\chi_{\Omega_{\alpha,d}}}(u,\etab) = \alpha^{-(d-1)} F_{d-1}\left(u,\frac{|\etab|}{\alpha}\right),
\]
for some two-variable function $F_{d-1}(u,s)$. Thus real zeros of $\widehat{\chi_{\Omega_{\alpha,d}}}$ coincide with the real zeros of $F_{d-1}$. The key point is that, for $d\geq 3$, there are no zeros on the long symmetry axis: indeed, by Lemma~\ref{lem:axial},
\[
F_{d-1}(u,0)>0 \qquad \text{for all } u\in\mathbb R.
\]
This positivity is then extended away from the axis by Lemma~\ref{lem:offaxis}: there exists $\sigma_d>0$ such that
\[
F_{d-1}(u,s)>0 \qquad \text{for all } u\in\mathbb R \text{ and all } 0\leq s\leq \sigma_d.
\]
Consequently every real zero must satisfy
\[
|\etab|>\alpha\sigma_d,
\]
which gives the lower bound $\kappa(\Omega_{\alpha,d}) \ge \alpha \sigma_d$. On the other hand,
\[
\kappa(\Omega_{\alpha,d}^{\ast}) = c_d\alpha^{(d-1)/d}.
\]
Comparing these two bounds yields Theorem~\ref{thm:highd}.

The proof of Theorem~\ref{thm:monot} uses the fact that for regular polygons $P_{2n}$, the location of the first zero of $\widehat{\chi_{P_{2n}}}$ can be determined exactly.
Namely, if $n$ is odd, then the zeros of $\fhat{P_{2n}}$ nearest to the origin lie on rays through the vertices of $P_{2n}$, and
if $n$ is even, then they lie on rays through the side midpoints.

We fix some additional notation.
We write $\ev\phi:=(\cos\phi,\sin\phi)\in\mathbb{S}^1$ for the unit vector pointing in the direction
$\phi$, both in the domain space and in the Fourier space, and use the polar
representations $\xb=r\ev\theta$ ($r=|\xb|\ge0$) and $\xib=\rho\ev\phi$ ($\rho=|\xib|\ge
0$).  We denote, for a unit vector $\ev{}$, the first positive zero of $\rho\mapsto\fhat{\Omega}(\rho\,\ev{})$ by  $\kappa_1(\ev{})=\kappa_1(\ev{};\Omega)$; thus, 
\begin{equation}\label{eq:kappa_as_min_over_e}
\kappa(\Omega)=\min_{\ev{}\in\mathbb{S}^1} \kappa_1(\ev{};\Omega).
\end{equation}
Here $\kappa_1(\ev{})$ is finite for every $\ev{}$ by \eqref{eq:BLP48} below, and the
infimum is attained because $\N(\Omega)$ is a nonempty closed set.
We call $\ev{}\in \mathbb{S}^1$ a \emph{minimising direction}
if $\kappa(\Omega)\ev{}\in\N(\Omega)$, i.e.\ if
$\kappa_1(\ev{}; \Omega)=\kappa(\Omega)$.

Let $n \geq 2$. We position the regular polygon $P_{2n}$, centred at the origin, in such a way that $\ev{0}$ is orthogonal to one of its sides, see Figure~\ref{fig:poly}.
Set
\[
\delta:=\frac{\pi}{2n},
\]
and denote by
\begin{equation}\label{eq:VM}
\mathcal{V}:=\left\{\ev{(2j-1)\delta},\ j=1,\dots,2n\right\},
\qquad
\mathcal{M}:=\left\{\ev{2j\delta},\ j=0,\dots,2n-1\right\},
\end{equation}
the sets of \emph{vertex directions} and of \emph{midpoint directions}, respectively: these are the directions of the vertices and of the side midpoints of $P_{2n}$, and, being sets of directions, they are unchanged under scaling.

\begin{theorem}\label{thm:regular_polygon_active_direction}
For the regular polygon $P_{2n}$, the set of minimising directions is the set $\mathcal{M}$ of side midpoint directions if $n$ is even, and the set $\mathcal{V}$ of vertex directions if $n$ is odd.
\end{theorem}

\begin{example}\label{ex:square}
Theorem~\ref{thm:regular_polygon_active_direction} holds for $n=2$ as follows from an explicit calculation. Here $\delta=\frac{\pi}{4}$, and the square of area $\pi$ is
\[
Q=P_4=\left\{(x_1,x_2):\ |x_1|<s,\ |x_2|<s\right\},
\qquad
s=\frac{\sqrt\pi}{2},
\]
with the sides perpendicular to the axes. Then, with $\xib=(\xi_1,\xi_2)$,
\begin{equation}\label{eq:Fsquare}
\fhat{Q}(\xib)=
\begin{cases}
\frac{4\sin(s\xi_1)\sin(s\xi_2)}{\xi_1\xi_2}\qquad&\text{if }\xi_1,\xi_2\ne 0,\\
2s\cdot\frac{2\sin(s\xi_j)}{\xi_j}\qquad&\text{if }\xi_i=0\text{ and }\xi_j\ne 0, i,j=1,2,\\
4s^2\qquad&\text{if }\xi_1=\xi_2=0.
\end{cases}
\end{equation}
Thus, 
\[
\N(Q)=\left\{\xib:\ |\xi_1|\in\frac{\pi}{s}\mathbb{N}\right\}\cup
\left\{\xib:\ |\xi_2|\in\frac{\pi}{s}\mathbb{N}\right\},
\]
and $\kappa(Q)=\pi/s=2\sqrt\pi=3.5449\ldots$, attained exactly at the four points
$\pm\frac{\pi}{s}\ev0$, $\pm\frac{\pi}{s}\ev{\pi/2}$, i.e.\ exactly in the four
midpoint directions. This is the two-dimensional case of \cite[Example~3.6]{BLP09}.
Since $n=2$ is even, this is precisely the assertion of
Theorem~\ref{thm:regular_polygon_active_direction}.
\end{example}

Theorem~\ref{thm:regular_polygon_active_direction} reduces the proof of Theorem~\ref{thm:monot} to comparing zeros that are given explicitly.
Even though they only hold with small margins, the comparisons 
\[
\kappa(P_{12}) > \kappa(\D) > \kappa(P_{10}) >\kappa(P_8) > \kappa(P_6)> \kappa(P_4)
\]
become straightforward.
The infinite chain of comparisons $\kappa(P_{2n}) > \kappa(P_{2n+2})$ for $n \geq 6$ requires an asymptotic estimate of the form $\kappa(P_{2n})=j_{1,1}+C n^{-6}+O(n^{-8})$ for an explicit constant $C$ and an explicit bound on the $O(n^{-8})$-term.
Again, thanks to Theorem~\ref{thm:regular_polygon_active_direction}, we obtain an explicit expression for $\kappa(P_{2n})$ and use standard estimates together with certified numerics to bound an infinite tail in this expression.

\subsection{A note on computer-assisted proofs}\label{ssec:comp}

In this paper, several  results, including Theorems~\ref{thm:monot} and~\ref{thm:effective}, have some computer-assisted component. In recent years, these techniques have been more widely successful in mathematics. In the context of spectral problems, similar computational techniques were very recently applied in 
\cite{Filonov-Levitin-Polterovich-Sher:polya-conjecture-balls,Filonov-Levitin-Polterovich-Sher:polya-conjecture-higher-dimensional-neumann-balls,GHLZ2026,DGP2026} where rigorous asymptotic analysis of spectral properties was performed. We refer to the survey~\cite{GomezSerrano:survey-cap-in-pde} for a more specific
treatment of computer-assisted proofs in PDEs.

The main idea is to substitute
floating point computations by rigorous upper and lower bounds, which are then propagated through every computer instruction and augmented if necessary by
taking into account any error made throughout the process. In our
concrete case, we will use computer-assisted verification specifically in the proof of Propositions~\ref{prop:finite-kappa-comparison} and~\ref{prop:bridge}, and Lemmas~\ref{lem:certified} and~\ref{lem:elemseries}. All of these verified intervals refer to a finite number of evaluations of expressions involving special functions. In all cases, the implementation was either direct or via a branch and bound method. The implementation is described in \S\ref{ssec:certified}.

\subsection{Organisation of the paper} 

In \S\ref{section:unboundedness-in-dim-geq-3} we give a proof of \eqref{eq:highd_main}, in particular constructing counterexamples to \cite[Conjecture 2.2]{BLP09} in dimension $d \geq 3$.
In \S\ref{section:nearest-zero} we prove the crucial property that the zeros of $\widehat{\chi_{P_{2n}}}$ nearest to the origin occur at rays through vertices or side midpoints, depending on the parity of $n$, see Theorem~\ref{thm:regular_polygon_active_direction}.
We then use this property to prove another of our main results:
in \S\ref{section:regular-polygons-in-the-plane}, we prove a statement about the values of $\kappa(P_{2n})$, in particular proving that the $12$-gon is a counterexample to \cite[Conjecture 2.2]{BLP09} in dimension $d=2$, see Theorem~\ref{thm:monot}.
Appendices~\ref{app:num} and~\ref{app:bessel} contain the proof of \eqref{eq:master}, the certified constants and elementary bounds used in \S\ref{section:regular-polygons-in-the-plane}, and a toolbox of Bessel function estimates.

\section{Unboundedness of $\kappa$ in dimension $d \geq 3$}\label{section:unboundedness-in-dim-geq-3}

Throughout this section we fix $d\ge 3$, write $\ell:=d-1\ge 2$, and use coordinates $\xib=(u,\etab)\in\R\times\R^{d-1}$.
The set $\Omega_{\alpha,d}$ defined by \eqref{eq:bipyramid} is open, centrally symmetric, and convex (sublevel set of the norm $(x,\yb)\mapsto |x|+\alpha|\yb|$). Slicing perpendicular to the first coordinate axis, the cross-section at height $x\in(-1,1)$ is the $(d-1)$-ball of radius $(1-|x|)/\alpha$. Hence
\begin{equation}\label{eq:vol}
\left|\Omega_{\alpha,d}\right|
= 2\int_0^1 \omega_{\ell}\left(\frac{1-x}{\alpha}\right)^{\ell}\!\dd x
= \frac{2\omega_{\ell}}{d}\,\alpha^{-\ell},
\end{equation}
where $\omega_\ell$ is the volume of the unit ball $B_\ell$. Therefore the equal-volume ball $\Omega_{\alpha,d}^{\ast}$ has radius
\begin{equation}\label{eq:Rad}
R_{\alpha,d}
=\left(\left|\Omega_{\alpha,d}\right|/\omega_d\right)^{1/d}
=\left(\frac{2\omega_{\ell}}{d\,\omega_d}\right)^{\!1/d}\alpha^{-\ell/d},
\end{equation}
and consequently
\begin{equation}\label{eq:kappa_star}
\kappa(\Omega_{\alpha,d}^{\ast})
=\frac{j_{d/2,1}}{R_{\alpha,d}}
= j_{d/2,1}\left(\frac{d\,\omega_d}{2\omega_{\ell}}\right)^{\!1/d}\alpha^{\ell/d}
\end{equation}
by \cite[Lemma 3.2 and Example 3.5]{BLP09}.
In particular $\kappa(\Omega_{\alpha,d}^{\ast})$ grows like $\alpha^{(d-1)/d}$ as $\alpha\to\infty$.

For $\xib=(u,\etab)\in\R\times\R^{d-1}$, slicing perpendicular to the first axis and using
$\int_{B_{\ell}(r)}\er^{\ir\etab\cdot\yb}\dd \yb=r^{\ell}\mathcal{B}_{\ell}(r|\etab|)$, where
\[
\mathcal{B}_{\ell}(t):=\int_{B_{\ell}}\er^{\ir t z_1}\dd \zb=(2\pi)^{\ell/2}\,\frac{J_{\ell/2}(t)}{t^{\ell/2}},\qquad
\mathcal{B}_{\ell}(0)=\omega_{\ell},
\]
we obtain
\begin{equation}\label{eq:fhat_alphad}
\begin{split}
\widehat{\chi_{\Omega_{\alpha,d}}}(u,\etab)
&=
\int_{-1}^1 \er^{\ir ux} \left( \frac{1-|x|}{\alpha} \right)^\ell \mathcal{B}_\ell \left((1-|x|) \frac{|\etab|}{\alpha} \right) \dd x
\\
&=
2\alpha^{-\ell}
\int_0^1 \cos(ux)(1-x)^\ell \mathcal{B}_\ell \left((1-x) \frac{|\etab|}{\alpha} \right) \dd x
\\
&=
\frac{1}{\alpha^{\ell}}\,F_{\ell}\left(u,\frac{|\etab|}{\alpha}\right)
\end{split}
\end{equation}
where 
\[
F_{\ell}(u,s):=2\int_0^1 \cos(ux)\,(1-x)^{\ell}\,\mathcal{B}_{\ell}\left(s(1-x)\right)\dd x.
\]
Consequently a real zero of $\widehat{\chi_{\Omega_{\alpha,d}}}$ at $(u,\etab)$ corresponds exactly to a real zero of $F_{\ell}$ at $(u,s)$ with $s=|\etab|/\alpha$.

For $\ell\ge 0$ and $u\in\R$, set
\begin{equation}\label{eq:Im_def}
I_{\ell}(u):=\int_0^1 (1-x)^{\ell}\cos(ux)\dd x,
\end{equation}
so that $F_{\ell}(u,0)=2\omega_{\ell}\,I_{\ell}(u)$.

\begin{lemma}\label{lem:axial}
For every integer $\ell\ge 2$,
\begin{equation}\label{eq:axial_pos}
I_{\ell}(u)>0\qquad\text{for all }u\in\R.
\end{equation}
Equivalently, $F_{\ell}(u,0)>0$ for every $u\in\R$ whenever $\ell=d-1\ge 2$.
\end{lemma}

\begin{proof}
At $u=0$, $I_{\ell}(0)=1/(\ell+1)>0$. For $u\ne 0$, two integrations by parts in \eqref{eq:Im_def} yield, since the boundary terms vanish at $x=1$ (factor $(1-x)$) and the sine boundary term vanishes at $x=0$,
\begin{equation}\label{eq:Im_rec}
u^{2}\,I_{\ell}(u)=\ell\left(1-(\ell-1)\,I_{\ell-2}(u)\right),\qquad \ell\ge 2.
\end{equation}
For $u\ne 0$, we have $\cos(ux)<1$ on a subset of $[0,1]$ of positive Lebesgue measure, hence
\begin{equation}\label{eq:Im_bound}
I_{\ell-2}(u)\;<\;\int_0^1(1-x)^{\ell-2}\dd x=\frac{1}{\ell-1}.
\end{equation}
Combined with \eqref{eq:Im_rec} this gives $u^{2}I_{\ell}(u)>0$, hence $I_{\ell}(u)>0$.
\end{proof}

\begin{remark}\label{rem:worksd3}
The inequality \eqref{eq:axial_pos} fails for $\ell=0,1$: $I_0(u)=\sin u/u$ vanishes at $u=k\pi$, $k\ne 0$, and $I_1(u)=(1-\cos u)/u^2$ vanishes at $u=2k\pi$, $k\ne 0$. Hence the threshold $\ell\ge 2$ --- i.e.\ $d\ge 3$ --- is sharp: it is exactly the regime in which the axial Fourier trace of $\Omega_{\alpha,d}$ stays uniformly above zero.
\end{remark}

The next lemma extends Lemma~\ref{lem:axial} from $s=0$ to a small interval $s\in[0,\sigma_d]$.

\begin{lemma}\label{lem:offaxis}
For every $d\ge 3$ there exists $\sigma_d>0$ such that
\begin{equation}\label{eq:offaxis}
F_{d-1}(u,s)>0\qquad\text{for every }u\in\R\ \text{and}\ s\in[0,\sigma_d].
\end{equation}
\end{lemma}

\begin{proof}
Write $\ell=d-1\ge 2$ and $f_{s}(x):=(1-x)^{\ell}\mathcal{B}_{\ell}(s(1-x))$, so that $F_{\ell}(u,s)=2\int_0^1\cos(ux)f_{s}(x)\dd x$.
Two integrations by parts (justified by $f_{s}(1)=0$ for $\ell\ge 1$ and $f_{s}'(1)=0$ for $\ell\ge 2$, both of which apply since $\ell\ge 2$) give
\begin{equation}\label{eq:u2F}
u^{2} F_{\ell}(u,s)=-2f_{s}'(0)-2\int_0^1 f_{s}''(x)\cos(ux)\dd x,\qquad u\ne 0.
\end{equation}
Using the Bessel identity $\mathcal{B}_{\ell}'(t)=-t\,\mathcal{B}_{\ell+2}(t)/(2\pi)$ derived from \cite[Equation 10.6.6]{DLMF},
\begin{equation}\label{eq:g_prime_0}
-f_{s}'(0)=\ell\,\mathcal{B}_{\ell}(s)-\frac{s^{2}}{2\pi}\,\mathcal{B}_{\ell+2}(s),
\end{equation}
which is a continuous function of $s$ with value $\ell\,\omega_{\ell}>0$ at $s=0$.

By the Riemann--Lebesgue lemma applied to $f_{s}''\in L^{1}[0,1]$, $\int_0^1 f_{s}''(x)\cos(ux)\dd x\to 0$ as $|u|\to\infty$ for each fixed $s$. To make this decay uniform in $s\in[0,s_{0}]$ for some $s_{0}>0$ we use one further integration by parts when $\ell\ge 3$ (which is justified because then $f_{s}''(1)=0$), giving
\begin{equation}\label{eq:residual}
\int_0^1 f_{s}''(x)\cos(ux)\dd x = -\frac{1}{u}\int_0^1 f_{s}'''(x)\sin(ux)\dd x,
\end{equation}
hence
\begin{equation}\label{eq:residual_bound}
\left|\int_0^1 f_{s}''(x)\cos(ux)\dd x\right|\le\frac{\|f_{s}'''\|_{L^{1}[0,1]}}{|u|}.
\end{equation}
The numerator is uniformly bounded in $s\in[0,s_0]$: indeed, $(x,s)\mapsto f_{s}'''(x)$ is jointly continuous on the compact set $[0,1]\times[0,s_0]$ (since $\mathcal{B}_\ell$ is entire), hence bounded there, so $\sup_{s\in[0,s_0]}\|f_{s}'''\|_{L^1[0,1]}<\infty$. (For $\ell=2$, an extra boundary term $f_s''(1)\sin u/u=2\omega_2\sin u/u$ appears, uniformly bounded by $2\omega_2/|u|$ independently of $s$, and the same conclusion holds.)

Combining \eqref{eq:u2F}--\eqref{eq:residual_bound}, there exist $U_d>0$ and $\sigma_d^{(1)}>0$ such that
\begin{equation}\label{eq:large_u}
\begin{split}
&\quad
u^{2}F_{\ell}(u,s)
\\
&\geq
-2f_{s}'(0)- \left| 2\int_0^1 f_{s}''(x)\cos(ux)\dd x \right|
\\
&\geq
\frac{3}{2}
\ell\,\mathcal{B}_{\ell}(0)
- \left| 2\int_0^1 f_{s}''(x)\cos(ux)\dd x \right|
\qquad
\text{by \eqref{eq:g_prime_0} for $s$ small}
\\
&\geq
\frac{3}{2}
\ell\,\mathcal{B}_{\ell}(0)
- 2 \cdot \frac{1}{4} \ell \omega_\ell
\qquad
\text{by \eqref{eq:residual_bound} for $|u|$ large}
\\
&= \ell\omega_{\ell}
\end{split}
\end{equation}
for all $|u|\ge U_d\ \text{and}\ s\in[0,\sigma_d^{(1)}]$.
For $|u|\le U_d$, $F_\ell(\cdot,0)$ is continuous and strictly positive on the compact interval $[-U_d,U_d]$ (Lemma~\ref{lem:axial}), so it admits a positive minimum $c_d>0$; the joint continuity of $F_\ell$ on $[-U_d,U_d]\times[0,s_0]$ (a parameter integral whose integrand is jointly continuous and dominated by $\sup_{[0,s_0]}|\mathcal{B}_\ell|$) then yields a $\sigma_d^{(2)}>0$ with $F_\ell(u,s)\ge c_d/2$ for all $(u,s)\in[-U_d,U_d]\times[0,\sigma_d^{(2)}]$.

Setting $\sigma_d:=\min(\sigma_d^{(1)},\sigma_d^{(2)})$ yields \eqref{eq:offaxis}.
\end{proof}

\begin{remark}\label{rem:quant_sigma}
Using the proof above, one can give an explicit value of
$\sigma_d>0$. There are three steps to make quantitative. First, the
identity
$\mathcal B_\ell'(s)
=-\frac{s}{2\pi}\mathcal B_{\ell+2}(s)$
\cite[Equation 10.6.6]{DLMF}, together with Lemma~\ref{lem:poisson}, gives
explicit bounds for
$|\mathcal B_\ell(s)-\mathcal B_\ell(0)|$ and
$|\mathcal B_{\ell+2}(s)|$ for small $s$, and hence an explicit
$\sigma_d^{(1)}\in(0,1]$. Substitution in \eqref{eq:g_prime_0} gives
the second inequality in \eqref{eq:large_u}.

Second, differentiating $f_s$ three times and using the same Bessel
identities and bounds gives an explicit estimate, uniform for
$0\le s\le1$, of $\|f_s'''\|_{L^1[0,1]}$. For $d\ge4$,
\eqref{eq:residual_bound}, and for $d=3$ the corresponding
integration-by-parts estimate with its additional explicit boundary
term, then give the third inequality in \eqref{eq:large_u} for an
explicit value of $U_d\ge1$.

Finally, rewriting \eqref{eq:Im_rec} as an integral involving
$(1-\cos(ux))/u^2$, with the continuous interpretation at $u=0$,
and restricting this integral to $0\le x\le1/U_d$, where
$1-\cos(ux)\ge u^2x^2/4$, gives an explicit lower bound for
$F_\ell(u,0)$ on $|u|\le U_d$. The preceding bound for
$|\mathcal B_\ell(s)-\mathcal B_\ell(0)|$ gives an explicit modulus
of continuity for $F_\ell(u,s)$, uniform in $u$.
This gives an explicit
$\sigma_d^{(2)}$.
We do not record the resulting constants here.
\end{remark}

\begin{proof}[Proof of Theorem~\ref{thm:highd}]
Let $\xib=(u,\etab)\in\R^d\setminus\{0\}$ be a real zero of $\widehat{\chi_{\Omega_{\alpha,d}}}$. By the sentence following \eqref{eq:fhat_alphad}, $F_{d-1}(u,|\etab|/\alpha)=0$, and by Lemma~\ref{lem:offaxis} this forces $|\etab|/\alpha>\sigma_d$, i.e.\ $|\etab|>\alpha\sigma_d$. Hence
\begin{equation}\label{eq:kappa_lower}
\kappa(\Omega_{\alpha,d})\ge \inf_{\xib:\widehat{\chi}(\xib)=0}|\xib|\ge \alpha\,\sigma_d.
\end{equation}
On the other hand, by \eqref{eq:kappa_star}, $\kappa(\Omega_{\alpha,d}^{\ast})=c_d\,\alpha^{(d-1)/d}$ with $c_d:=j_{d/2,1}(d\omega_d/(2\omega_{\ell}))^{1/d}$.
Comparing,
\begin{equation}\label{eq:ratio}
\frac{\kappa(\Omega_{\alpha,d})}{\kappa(\Omega_{\alpha,d}^{\ast})}
\ge \frac{\sigma_d}{c_d}\,\alpha^{1/d}\;\longrightarrow\;\infty\qquad\text{as }\alpha\to\infty,
\end{equation}
which yields both \eqref{eq:highd_main} (with $\alpha_d^{\ast}:=(c_d/\sigma_d)^{d}$) and the explicit rate \eqref{eq:highd_rate}.
\end{proof}

\begin{remark}\label{rmk:highd_numerics}
The threshold $\alpha_d^{\ast}$ is well within reach in every dimension. Non-rigorous numerical evaluation of $\kappa(\Omega_{\alpha,d})$ via the smallest real zero of $F_{d-1}(u,s)$ gives, at $\alpha=100$, the ratios
\[
\begin{array}{c|cccccccc}
d & 3 & 4 & 5 & 6 & 7 & 8 & 9 & 10\\\hline
\kappa(\Omega_{\alpha,d})/\kappa(\Omega_{\alpha,d}^{\ast}) & 1.803 & 1.563 & 1.382 & 1.276 & 1.210 & 1.165 & 1.133 & 1.110
\end{array}
\]
Rigorous interval-arithmetic certification for any fixed $(\alpha,d)$ reduces to bounding $F_{d-1}(u,s)$ uniformly on the 2D region $\{u^{2}+(\alpha s)^{2}\le \kappa(\Omega_{\alpha,d}^{\ast})^{2}\}$ and is left to future work.
\end{remark}

\section{The nearest-zero direction for regular polygons: proof of Theorem~\ref{thm:regular_polygon_active_direction}}\label{section:nearest-zero}

\subsection{Notation and scaling for regular polygons}\label{ssec:polynotation}

Throughout this section and \S\ref{section:regular-polygons-in-the-plane}, $P_{2n}$, $n\ge2$, denotes the regular $2n$-gon of area $\pi$, centred at the origin and positioned as in \S\ref{ssec:ideas}, so that $\ev0$ is orthogonal to one of its sides; $\delta=\frac{\pi}{2n}$, and $\mathcal{V}$ and $\mathcal{M}$ are the sets of vertex and of side midpoint directions \eqref{eq:VM}.

By \eqref{eq:scaleFT}, any $\Omega$ and its rescaled copy $s\Omega$ have the same minimising directions. It will be convenient to work with the rescaled polygon of circumradius one, and we adopt throughout the following convention: a tilde marks a quantity attached to the circumradius-one polygon $\widetilde{P}_{2n}$, and the same letter without a tilde marks the corresponding quantity attached to the area-$\pi$ polygon $P_{2n}$. 

\begin{figure}[htb]
\centering
\includegraphics{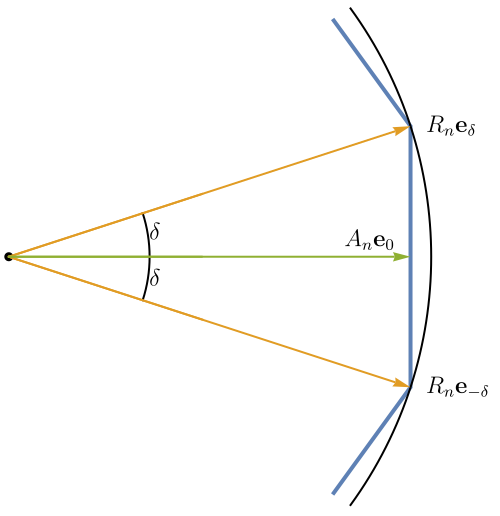}
\caption{The geometry of a regular polygon $P_{2n}$.}
\label{fig:poly}
\end{figure}

Thus $\widetilde{P}_{2n}$ is the regular $2n$-gon of circumradius one, centred at the origin, with the vertices at the points of $\mathcal{V}$ and the side midpoints at the points of $(\cos\delta)\,\mathcal{M}$. Its area is 
\[
\left|\widetilde{P}_{2n}\right|=n\sin2\delta, 
\]
so, writing $R_n$ for the dilation factor for which $P_{2n}=R_n\,\widetilde{P}_{2n}$, and using $2n\delta=\pi$, we get
\begin{equation}\label{eq:areapi-geom}
R_n=\sqrt{\pi/(n\sin2\delta)}=\sqrt{\frac{2\delta}{\sin2\delta}}
\qquad\text{and}\qquad
A_n:=R_n\cos\delta=\sqrt{\frac{\delta}{\tan\delta}}
\end{equation}
for the circumradius and the apothem of $P_{2n}$; the vertices of $P_{2n}$ are the points of $R_n\mathcal{V}$, and its side midpoints are the points of $A_n\mathcal{M}$. We record the identity and the inequalities
\begin{equation}\label{eq:areapi-geom2}
R_n^{2}-A_n^{2}=\delta\tan\delta,
\qquad
A_n<1<R_n,
\end{equation}
the latter following from $\tan\delta>\delta>\sin\delta\cos\delta$.

We define the \emph{radial function} $\widetilde{\mathcal{R}}(\theta)$ of $\widetilde{P}_{2n}$ by $\widetilde{P}_{2n}=\left\{r\,\ev\theta:\ 0\le r<\widetilde{\mathcal{R}}(\theta)\right\}$, thus by
\begin{equation}\label{eq:radialfn}
\begin{split}
\widetilde{\mathcal{R}}(\theta)&=\frac{\cos\delta}{\cos\theta}\qquad\text{if }\theta\in[-\delta,\delta],\\
\widetilde{\mathcal{R}}(\theta\pm 2\delta)&=\widetilde{\mathcal{R}}(\theta).
\end{split}
\end{equation}
The radial function is even, $(2\delta)$-periodic, and strictly increasing on $[0,\delta]$, from $\widetilde{\mathcal{R}}(0)=\cos\delta$ (side midpoint) to $\widetilde{\mathcal{R}}(\delta)=1$ (vertex). Accordingly, the radial function of $P_{2n}$ is
\begin{equation}\label{eq:radialfn-scaled}
\mathcal{R}:=R_n\widetilde{\mathcal{R}},
\qquad\text{with}\qquad
\mathcal{R}(\theta)=\frac{A_n}{\cos\theta}\quad\text{for }\theta\in[-\delta,\delta],
\end{equation}
extended evenly and $(2\delta)$-periodically.

Since $P_{2n}$ and $\widetilde{P}_{2n}$ have the same minimising directions, Theorem~\ref{thm:regular_polygon_active_direction} is equivalent to the assertion that the set of minimising directions of $\widetilde{P}_{2n}$ is $\mathcal{M}$ if $n$ is even, and $\mathcal{V}$ if $n$ is odd; it is the latter that we prove.

\subsection{Plan of the proof} 

As the case $n=2$ is covered by Example~\ref{ex:square}, we fix from now on $n\ge 3$ and write for brevity $\widetilde{P}:=\widetilde{P}_{2n}$ and $\widetilde{\kappa}:=\kappa\left(\widetilde{P}\right)$. 

Recall that the support function of a convex $\Omega$ is
$w_\Omega(\mathbf e)=\sup_{\xb\in\Omega}\xb\cdot\mathbf e$; for balanced
$\Omega$ it is the half-width of $\Omega$ in the direction $\mathbf e$, see
Figure~\ref{fig:geom}.

\begin{figure}[htb]
\centering
\includegraphics[width=0.6\textwidth]{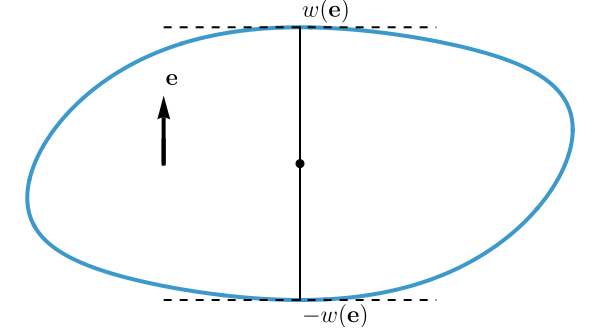}
\caption{The support function of $\Omega$.}
\label{fig:geom}
\end{figure}

For any convex balanced planar domain $\Omega\subset\mathbb{R}^2$, we have, by \cite[Lemma~4.8]{BLP09}, see also \cite{Zastavnyi},
\begin{equation}\label{eq:BLP48}
\kappa_1(\ev{})\ \le\ \frac{2\pi}{w(\ev{})}
\end{equation}
for any $\ev{}\in \mathbb{S}^1$. Since $\widetilde{P}$ has circumradius one, taking any $\ev{}\in\mathcal{V}$ gives $w_{\widetilde{P}}(\ev{})=1$, 
and therefore by \eqref{eq:kappa_as_min_over_e}, 
\begin{equation}\label{eq:kappa2pi}
0<\widetilde{\kappa}\le2\pi.
\end{equation}

Then,  
$\fhat{\widetilde{P}}(\rho\,\ev\phi)>0$ for all $0<\rho<\widetilde{\kappa}$ and all
$\phi$, and $\fhat{\widetilde{P}}(\widetilde{\kappa}\,\ev\phi)\ge0$ for all $\phi$.  If $\ev{\phi^\star}$ is a minimising direction, then $\fhat{\widetilde{P}}(\widetilde{\kappa}\,\ev{\phi^\star})=0$, and
$\phi^\star$ is a global minimiser of 
\[
\widetilde{\Phi}_\rho(\phi):=\fhat{\widetilde{P}}(\rho\,\ev\phi),\qquad\phi\in\mathbb{R},
\]
the minimum value being $0$; moreover $\kappa_1(\ev{\phi^\star})=\widetilde{\kappa}$.

Theorem~\ref{thm:regular_polygon_active_direction} will follow once we know, for
$\rho=\widetilde{\kappa}$, that the global minimisers of $\widetilde{\Phi}_\rho(\phi)$ are
exactly the side midpoint directions ($n$ even) resp.\ the vertex directions ($n$ odd).  We shall prove this for \emph{every} $\rho\in(0,2\pi]$, which suffices by
\eqref{eq:kappa2pi}.

To proceed, we state 
\begin{proposition}\label{prop:fourier}
For all $\rho>0$ and $\phi\in\R$,
\begin{equation}\label{eq:fourier}
\widetilde{\Phi}_\rho(\phi)=\fhat{\widetilde{P}}(\rho\,\ev\phi)=\widetilde{a}_0(\rho)+2\sum_{m\ge1}(-1)^{nm}\widetilde{b}_m(\rho)\cos(2nm\,\phi),
\end{equation}
where 
\begin{equation}\label{eq:Bm}
\widetilde{a}_0(\rho)=\int_{\widetilde{P}} J_0(\rho|\xb|)\,\dr \xb,\qquad
\widetilde{b}_m(\rho)=\int_{\widetilde{P}} J_{2nm}(\rho|\xb|)\cos\left(2nm\theta\right)\,\dr \xb,\qquad \dr \xb = r\,\dr r\,\dr\theta,
\end{equation}
and the series \eqref{eq:fourier} together with all its $\phi$-derivatives converges
absolutely and uniformly on compact sets of $(\rho,\phi)$.  In particular,
the function $\widetilde{\Phi}_\rho(\phi)$ is $C^\infty$, even, and
$\frac{\pi}{n}$-periodic, and the series may be differentiated in $\phi$ term-wise.
\end{proposition}

It will be convenient to substitute
\begin{equation}\label{eq:psi}
\begin{split}
\psi&:=2n\phi,\\
\widetilde{\Psi}_\rho(\psi)&:=\widetilde{\Phi}_\rho\left(\frac{\psi}{2n}\right)
=\widetilde{a}_0(\rho)+\sum_{m=1}^\infty \widetilde{c}_m(\rho)\cos(m\psi),\\
\widetilde{c}_m(\rho)&:=2(-1)^{nm}\widetilde{b}_m(\rho).
\end{split}
\end{equation}

The side midpoint directions correspond to $\psi\equiv0\pmod{2\pi}$ and the vertex
directions to $\psi\equiv\pi\pmod{2\pi}$; the fundamental arc $\phi\in[0,\delta]$,
which runs from a side midpoint ($\phi=0$) to a vertex ($\phi=\delta$), corresponds to
$\psi\in[0,\pi]$, and every direction $\phi\in\mathbb{R}$ is equivalent, under the symmetries of
$\widetilde{\Psi}_\rho$ (which is even and $2\pi$-periodic in $\psi$), to exactly one
$\psi\in[0,\pi]$.

Further on, we have 
\begin{proposition}\label{prop:domination}
Fix $\rho>0$ and assume that
\begin{equation}\label{eq:dom}
\sum_{m=2}^\infty m^2|\widetilde{b}_m(\rho)|<-\widetilde{b}_1(\rho),
\end{equation}
and, therefore, 
\begin{equation}\label{eq:b1neg}
\widetilde{b}_1(\rho)<0.
\end{equation}
Then,
\begin{enumerate}
\item[\textup{(i)}] $\widetilde{\Psi}_\rho'(\psi)\neq 0$ for all $\psi\in(0,\pi)$. Moreover,
\[
\sgn \widetilde{\Psi}_\rho'(\psi)=-\sgn \widetilde{c}_1(\rho)=(-1)^n,\qquad \psi\in(0,\pi),
\]
so $\widetilde{\Psi}_\rho$ is strictly increasing on $[0,\pi]$ if $n$ is even and strictly decreasing if $n$ is odd.
\item[\textup{(ii)}] The set of global minimisers of
$\widetilde{\Phi}_\rho(\phi)$ over $\R$ is exactly the set of side midpoint
directions if $n$ is even, and exactly the set of vertex directions if $n$ is odd.
\end{enumerate}
\end{proposition}

It remains, therefore, to check that conditions \eqref{eq:dom}, \eqref{eq:b1neg} hold. This is done through a sequence of technical bounds.

Fix $n\ge3$ and $\rho\in(0,2\pi]$.  

\begin{lemma}\label{lem:upper}
For all $n\ge3$ and $0<\rho\le2\pi$,
\begin{equation}\label{eq:U}
\sum_{m=2}^\infty m^2 |\widetilde{b}_m(\rho)|\ \le\ U(\rho):=4.01\cdot\frac{2\pi}{4n+2}\cdot\frac{(\rho/2)^{4n}}{(4n)!}.
\end{equation}
\end{lemma}

We now bound $-\widetilde{b}_1(\rho)$ from below.  Define
\begin{equation}\label{eq:Phi}
\mathcal{T}_\rho(s):=\int_0^sJ_{2n}(\rho r)\,r\,\dr r,\qquad 0\le s\le1,
\end{equation}
which, by Lemma~\ref{lem:lorch} ($\rho s\le2\pi<j_{2n,1}$, so $J_{2n}(\rho r)>0$ for
$r\in(0,1]$), is strictly increasing on $[0,1]$, and set
\begin{equation}\label{eq:varrho12}
\widetilde{\mathcal{R}}_1:=\widetilde{\mathcal{R}}(3\delta/4)=\frac{\cos\delta}{\cos(3\delta/4)},\qquad
\widetilde{\mathcal{R}}_2:=\widetilde{\mathcal{R}}(\delta/4)=\frac{\cos\delta}{\cos(\delta/4)},
\qquad
\cos\delta\le\widetilde{\mathcal{R}}_2<\widetilde{\mathcal{R}}_1<1.
\end{equation}

\begin{lemma}\label{lem:pairing}
For all $n\ge3$ and $0<\rho\le2\pi$,
\begin{equation}\label{eq:pairing}
-\widetilde{b}_1(\rho)=4n\int_0^{\delta/2}
\left(\mathcal{T}_\rho\left(\widetilde{\mathcal{R}}(\frac\delta2+v)\right)
-\mathcal{T}_\rho\left(\widetilde{\mathcal{R}}(\frac\delta2-v)\right)\right)\sin(2nv)\,\dr v>0,
\end{equation}
and
\begin{equation}\label{eq:B1low}
-\widetilde{b}_1(\rho)\ \ge\ \frac{\pi\sqrt2}{4}\int_{\widetilde{\mathcal{R}}_2}^{\widetilde{\mathcal{R}}_1}J_{2n}(\rho r)\,r\,\dr r
\ \ge\ \frac{\pi\sqrt2}{4}\,(\widetilde{\mathcal{R}}_1-\widetilde{\mathcal{R}}_2)\,\widetilde{\mathcal{R}}_2\,
\min\left\{J_{2n}(\rho\widetilde{\mathcal{R}}_2),\,J_{2n}(\rho\widetilde{\mathcal{R}}_1)\right\}.
\end{equation}
\end{lemma}

\begin{corollary}\label{cor:lower}
For all $n\ge3$ and $0<\rho\le2\pi$,
\begin{equation}\label{eq:V}
-\widetilde{b}_1(\rho)\ \ge\ V(\rho):=
\frac{\pi\sqrt2}{4}\,(\widetilde{\mathcal{R}}_1-\widetilde{\mathcal{R}}_2)\,\widetilde{\mathcal{R}}_2^{\,2n+1}\,
\underline{\mathcal{J}}_{2n}\left(\pi^2\widetilde{\mathcal{R}}_1^2\right)\cdot\frac{(\rho/2)^{2n}}{(2n)!}\,,
\end{equation}
provided $\underline{\mathcal{J}}_{2n}(\pi^2\widetilde{\mathcal{R}}_1^2)>0$ \textup{(}which is verified in
Proposition~\ref{prop:master} below for every $n\ge3$, see also Appendix~\ref{ssec:proofmaster}\textup{)}; here $\underline{\mathcal{J}}_{2n}$ is the
polynomial of Lemma~\ref{lem:phi4} with $\nu=2n\ge 6$.
\end{corollary}

Finally, we have

\begin{proposition}\label{prop:master}
For every $n\ge3$,
\begin{equation}\label{eq:master}
\frac{\pi^{2n}\,(2n)!}{(4n)!}\;<\;
\mathcal{S}_n:=\frac{\sqrt2\,(4n+2)}{32.08}\;
(\widetilde{\mathcal{R}}_1-\widetilde{\mathcal{R}}_2)\,\widetilde{\mathcal{R}}_2^{\,2n+1}\,\underline{\mathcal{J}}_{2n}\left(\pi^2\widetilde{\mathcal{R}}_1^2\right),
\end{equation}
where $\widetilde{\mathcal{R}}_1,\widetilde{\mathcal{R}}_2$ are given by \eqref{eq:varrho12}. 
Consequently
\begin{equation}\label{eq:UV}
\sum_{m=2}^\infty m^2|\widetilde{b}_m(\rho)|\ \le\ U(\rho)\ <\ V(\rho)\ \le\ -\widetilde{b}_1(\rho)
\qquad\text{for all }0<\rho\le2\pi.
\end{equation}
\end{proposition}

Together, these results easily imply Theorem~\ref{thm:regular_polygon_active_direction}. 
By Proposition~\ref{prop:master}, hypothesis \eqref{eq:dom}, and, therefore, 
inequality \eqref{eq:b1neg}, hold for every $n\ge3$ and every $\rho\in(0,2\pi]$.  By \eqref{eq:kappa2pi},
we have $\widetilde\kappa\in(0,2\pi]$, so Proposition~\ref{prop:domination}(ii) applies at
$\rho=\widetilde\kappa$ and identifies the global minimisers of
$\widetilde\Phi_{\widetilde\kappa}$, and, equivalently, the minimising directions of $\widetilde P_{2n}$,  as $\mathcal M$ for
$n$ even and $\mathcal V$ for $n$ odd.  Since $P_{2n}$ and $\widetilde P_{2n}$
have the same minimising directions by \eqref{eq:scaleFT}, this proves Theorem~\ref{thm:regular_polygon_active_direction}  for
$n\ge3$; the case $n=2$ is Example~\ref{ex:square}.

\subsection{Proofs of Propositions~\ref{prop:fourier} and~\ref{prop:domination}}

\begin{proof}[Proof of Proposition~\ref{prop:fourier}]
We have, with $\xb=r\ev{\theta}$ and $\xib=\rho\ev{\phi}$, 
\[
\begin{split}
\fhat{\widetilde{P}}(\rho\,\ev\phi) &= \int_{\widetilde{P}} \cos(\xb\cdot\xib)\,\dr \xb
= \int_{\widetilde{P}} \cos\left(r\rho(\cos\theta\cos\phi+\sin\theta\sin\phi)\right)\,\dr \xb\\
&= \int_{\widetilde{P}} \cos\left(r\rho \cos(\theta - \phi)\right)\,\dr \xb.
\end{split}
\]

The Jacobi--Anger expansion, see  \cite[\S2.22]{Watson} or
\cite[\S10.12.3]{DLMF}, states that for all real $z,\alpha$,
\[
\cos(z\cos\alpha)=J_0(z)+2\sum_{k=1}^\infty (-1)^kJ_{2k}(z)\cos(2k\alpha).
\]
With $z=\rho r$ and $\alpha=\theta-\phi$ and the bound
$|J_\nu(z)|\le(z/2)^\nu/\nu!$ of Lemma~\ref{lem:poisson}, the series is
dominated on $\widetilde{P}$ (where $|\xb|\le1$) by $2\sum_k (\rho/2)^{2k}/(2k)!<\infty$, hence
converges uniformly on $\widetilde{P}$, and term-by-term integration over $\widetilde{P}$ is legitimate, yielding
\[
\fhat{\widetilde{P}}(\rho\,\ev\phi)=\widetilde{a}_0(\rho)+2\sum_{k=1}^\infty (-1)^k
\int_{\widetilde{P}} J_{2k}(\rho\,|\xb|)\cos\left(2k(\theta-\phi)\right)\,\dr \xb.
\]
By dihedral symmetry, the integrals vanish unless $k=nm$, $m\in\mathbb{N}$, which yields, after minor simplifications, \eqref{eq:fourier}--\eqref{eq:Bm}.  Finally, by Lemma~\ref{lem:poisson},
$|\widetilde{b}_m(\rho)|\le(\rho/2)^{2nm}\pi/(2nm)!$, so for each fixed $\rho$ the coefficient
sequence decays superexponentially, which proves the remaining assertions.
\end{proof}

\begin{proof}[Proof of Proposition~\ref{prop:domination}]

By Proposition~\ref{prop:fourier} we may differentiate term-wise:
\[
\widetilde{\Psi}_\rho'(\psi)=-\widetilde{c}_1(\rho)\sin\psi-\sum_{m=2}^\infty m \widetilde{c}_m(\rho)\sin m\psi.
\]
As for every $m\in\mathbb{N}$ and every $\psi\in[0,\pi]$, we have $|\sin m\psi|\le m\sin\psi$, and since $|\widetilde{c}_m(\rho)|=2|\widetilde{b}_m(\rho)|$, we have, by the assumption in the statement of the proposition, 
\[
\left|\sum_{m=2}^\infty m \widetilde{c}_m(\rho)\sin m\psi\right|
\le\left(\sum_{m=2}^\infty m^2|\widetilde{c}_m(\rho)|\right)\sin\psi
<|\widetilde{c}_1(\rho)|\sin\psi,
\]
hence 
\[
\begin{split}
\sgn(\widetilde{\Psi}_\rho'(\psi))&=\sgn\left(-\widetilde{c}_1(\rho)\sin\psi-\sum_{m=2}^\infty  m \widetilde{c}_m(\rho)\sin m\psi\right)
=\sgn\left(-\widetilde{c}_1(\rho)\sin\psi\right)\\
&=\sgn(-\widetilde{c}_1(\rho))=\sgn\left(-2(-1)^n \widetilde{b}_1(\rho)\right)=(-1)^n
\end{split}
\]
for all of $(0,\pi)$.  Thus, $\widetilde{\Psi}_\rho$ is strictly
increasing on $[0,\pi]$ for $n$ even and strictly decreasing for $n$ odd, which proves (i).

For (ii), suppose $n$ even (the odd case is similar).  Strict monotonicity gives
$\widetilde{\Psi}(0)<\widetilde{\Psi}(\psi)$ for every $\psi\in(0,\pi]$.  As $\widetilde{\Psi}$ is even and $2\pi$-periodic,
the global minimisers of $\widetilde{\Psi}$ on $\R$ are exactly $\{2\pi k:k\in\mathbb{Z}\}$. Going back to the variable $\phi$ by \eqref{eq:psi}, we conclude that global minimisers of $\widetilde{\Phi}_\rho(\phi)$ satisfy
$\phi= 0\pmod{\pi/n}$, that is, they are exactly the $2n$ side midpoint
directions.  For $n$ odd the minimisers are $\phi\equiv\delta\pmod{\pi/n}$, that is the $2n$ vertex directions.
\end{proof}

\subsection{Proofs of technical bounds}

\begin{proof}[Proof of Lemma~\ref{lem:upper}]
We use the radial moments
\begin{equation}\label{eq:moments}
\widetilde{\mu}_p:=\int_{\widetilde{P}}|\xb|^p\,\dr \xb
=\int_0^{2\pi} \int_0^{\widetilde{\mathcal{R}}(\theta)} r^{p+1}\,\dr r\,\dr \theta=
\frac{1}{p+2}\int_0^{2\pi}\left(\widetilde{\mathcal{R}}(\theta)\right)^{p+2}\, \dr\theta
\ \le\ \frac{2\pi}{p+2},\qquad p\ge0,
\end{equation}
where we used the fact that $\widetilde{\mathcal{R}}(\theta)\le 1$.

By Lemma~\ref{lem:poisson} and \eqref{eq:moments}, for every $m\ge2$,
\[
|\widetilde{b}_m(\rho)|\le\int_{\widetilde{P}}\left|J_{2nm}(\rho|\xb|)\right|\,\dr \xb
\le\frac{(\rho/2)^{2nm}}{(2nm)!}\,\widetilde{\mu}_{2nm}
\le\frac{2\pi}{2nm+2}\cdot\frac{(\rho/2)^{2nm}}{(2nm)!}
\le\frac{2\pi}{4n+2}\cdot\frac{(\rho/2)^{2nm}}{(2nm)!}\,.
\]
Next, $(2nm)!=(4n)!\,(4n+1)(4n+2)\cdots(2nm)\ge(4n)!\,(4n+1)^{2n(m-2)}$, so with
\[
\lambda:=\left(\frac{\rho/2}{4n+1}\right)^{2n},
\]
we have
\[
\frac{(\rho/2)^{2nm}}{(2nm)!}\le\frac{(\rho/2)^{4n}}{(4n)!}\,\lambda^{\,m-2}.
\]
Since $\rho/2\le\pi$ and $n\ge3$, $0<\frac{\pi}{4n+1}\le\frac{\pi}{13}<1$ and
$2n\ge6$ give $\lambda\le(\pi/13)^6<2\cdot10^{-4}$.  Finally, by the standard 
identities $\sum_{j\ge0}\lambda^j=\frac1{1-\lambda}$, $\sum_{j\ge0}j\lambda^j=\frac{\lambda}{(1-\lambda)^2}$,
$\sum_{j\ge0}j^2\lambda^j=\frac{\lambda(1+\lambda)}{(1-\lambda)^3}$, we estimate
\[
\sum_{m=2}^\infty m^2\lambda^{\,m-2}
=\sum_{j\ge0}(j+2)^2\lambda^{\,j}
=\frac4{1-\lambda}+\frac{4\lambda}{(1-\lambda)^2}+\frac{\lambda(1+\lambda)}{(1-\lambda)^3}
\le 4.0009+0.0009+0.0003<4.01
\]
for $\lambda\le2\cdot10^{-4}$.  Combining the bounds yields \eqref{eq:U}.
\end{proof}

\begin{proof}[Proof of Lemma~\ref{lem:pairing}]
Passing to polar coordinates in \eqref{eq:Bm} and integrating in the radial variable
first,
\[
\widetilde{b}_1(\rho)=\int_0^{2\pi}\mathcal{T}_\rho\left(\widetilde{\mathcal{R}}(\theta)\right)\cos(2n\theta)\,d\theta
=:\int_0^{2\pi}H(\theta)\,\dr\theta.
\]
Both $\widetilde{\mathcal{R}}(\theta)$ and $\cos(2n\theta)$ are even and have the period $2\delta=\pi/n$, hence
so is $H(\theta)$, and
\[
\int_0^{2\pi}H(\theta)\,\dr\theta = 4n\int_0^{\delta}H(\theta)\,\dr\theta.
\]
Substituting $\theta=\frac\delta2+v$, $v\in\left[-\frac\delta2,\frac\delta2\right]$, and using
$\cos\left(2n(\frac\delta2+v)\right)=\cos(\frac\pi2+2nv)=-\sin2nv$, we get
\[
\begin{split}
\int_0^\delta H(\theta)\,\dr\theta
&=-\int_{-\delta/2}^{\delta/2}\mathcal{T}_\rho\left(\widetilde{\mathcal{R}}\left(\frac\delta2+v\right)\right)\sin(2nv)\,\dr v
\\&=-\int_0^{\delta/2}\left(\mathcal{T}_\rho\left(\widetilde{\mathcal{R}}\left(\frac\delta2+v\right)\right)
-\mathcal{T}_\rho\left(\widetilde{\mathcal{R}}\left(\frac\delta2-v\right)\right)\right)\sin(2nv)\,\dr v.
\end{split}
\]
This proves the identity in \eqref{eq:pairing}.  

For
$v\in\left(0,\frac\delta2\right]$ we have $0\le\frac\delta2-v<\frac\delta2+v\le\delta$, hence
$\widetilde{\mathcal{R}}(\frac\delta2+v)>\widetilde{\mathcal{R}}(\frac\delta2-v)$ (as $\widetilde{\mathcal{R}}$ strictly increasing on
$[0,\delta]$), hence the bracket in \eqref{eq:pairing} is strictly positive
($\mathcal{T}_\rho$ strictly increasing); also $\sin2nv>0$ since $2nv\in(0,\frac\pi2]$.
Thus the integrand in \eqref{eq:pairing} is positive and continuous, and
$-\widetilde{b}_1(\rho)>0$.

For the quantitative bound, restrict the integral \eqref{eq:pairing} to
$v\in\left[\frac\delta4,\frac\delta2\right]$, thus
$2nv\in\left[\frac\pi4,\frac\pi2\right]$,  $\sin2nv\ge\sin\frac\pi4=\frac{\sqrt2}2$,
and $\frac\delta2-v\in\left[0,\frac\delta4\right]$,
$\frac\delta2+v\in\left[\frac{3\delta}4,\delta\right]$. By the monotonicity of $\widetilde{\mathcal{R}}$
and $\mathcal{T}_\rho$,
\[
\mathcal{T}_\rho\left(\widetilde{\mathcal{R}}(\frac\delta2+v)\right)-\mathcal{T}_\rho\left(\widetilde{\mathcal{R}}(\frac\delta2-v)\right)
\ \ge\ \mathcal{T}_\rho(\widetilde{\mathcal{R}}_1)-\mathcal{T}_\rho(\widetilde{\mathcal{R}}_2)
=\int_{\widetilde{\mathcal{R}}_2}^{\widetilde{\mathcal{R}}_1}J_{2n}(\rho r)\,r\,\dr r.
\]
Since the discarded part of the integrand is nonnegative,
\[
-\widetilde{b}_1(\rho)\ \ge\ 4n\cdot\frac{\delta}{4}\cdot\frac{\sqrt2}{2}\cdot
\int_{\widetilde{\mathcal{R}}_2}^{\widetilde{\mathcal{R}}_1}J_{2n}(\rho r)\,r\,dr
=\frac{\pi\sqrt2}{4}\int_{\widetilde{\mathcal{R}}_2}^{\widetilde{\mathcal{R}}_1}J_{2n}(\rho r)\,r\,dr,
\]
using $n\delta=\pi/2$.  Finally $r\ge\widetilde{\mathcal{R}}_2$ on the integration range, and by
Lemma~\ref{lem:unimodal} (applicable since
$[\rho\widetilde{\mathcal{R}}_2,\rho\widetilde{\mathcal{R}}_1]\subset(0,2\pi]\subset(0,j_{2n,1})$ by
Lemma~\ref{lem:lorch}) the minimum of $J_{2n}$ on $[\rho\widetilde{\mathcal{R}}_2,\rho\widetilde{\mathcal{R}}_1]$ is
attained at an endpoint.  This gives the second inequality in \eqref{eq:B1low}.
\end{proof}

\begin{proof}[Proof of Corollary~\ref{cor:lower}]
For $i\in\{1,2\}$ set $y_i:=(\rho\widetilde{\mathcal{R}}_i/2)^2$.  Then
$0<y_i\le\pi^2\widetilde{\mathcal{R}}_1^2<\pi^2$, so by Lemma~\ref{lem:phi4}(i)--(ii),
\[
J_{2n}(\rho\widetilde{\mathcal{R}}_i)=\frac{(\rho\widetilde{\mathcal{R}}_i/2)^{2n}}{(2n)!}\,\mathcal{J}_{2n}(y_i)
\ \ge\ \frac{(\rho\widetilde{\mathcal{R}}_2/2)^{2n}}{(2n)!}\,\underline{\mathcal{J}}_{2n}(y_i)
\ \ge\ \frac{(\rho\widetilde{\mathcal{R}}_2/2)^{2n}}{(2n)!}\,\underline{\mathcal{J}}_{2n}(\pi^2\widetilde{\mathcal{R}}_1^2),
\]
where we used $\widetilde{\mathcal{R}}_i\ge\widetilde{\mathcal{R}}_2$, $\mathcal{J}_{2n}\ge\underline{\mathcal{J}}_{2n}$ on $[0,\pi^2]$, and the
monotonicity of $\underline{\mathcal{J}}_{2n}$ together with $y_i\le\pi^2\widetilde{\mathcal{R}}_1^2$.  (If
$\underline{\mathcal{J}}_{2n}(\pi^2\widetilde{\mathcal{R}}_1^2)>0$, the first inequality also uses $\underline{\mathcal{J}}_{2n}(y_i)>0$, which
holds since $\underline{\mathcal{J}}_{2n}$ is decreasing.)  Insert this into \eqref{eq:B1low} and use
$(\widetilde{\mathcal{R}}_2)^{2n}\cdot\widetilde{\mathcal{R}}_2=\widetilde{\mathcal{R}}_2^{2n+1}$.
\end{proof}

\begin{proof}[Proof of Proposition~\ref{prop:master}]

We will show here that \eqref{eq:master} implies \eqref{eq:UV}. The proof of \eqref{eq:master} is postponed to the appendix.  By
\eqref{eq:U} and \eqref{eq:V},
\[
\frac{U(\rho)}{V(\rho)}
=\left(\frac \rho2\right)^{2n}\cdot
\frac{4.01\cdot\frac{2\pi}{4n+2}\cdot\frac{1}{(4n)!}}
{\frac{\pi\sqrt2}{4}(\widetilde{\mathcal{R}}_1-\widetilde{\mathcal{R}}_2)\widetilde{\mathcal{R}}_2^{2n+1}\underline{\mathcal{J}}_{2n}(\pi^2\widetilde{\mathcal{R}}_1^2)\cdot
\frac{1}{(2n)!}}\,,
\]
which is increasing in $\rho$ on $(0,2\pi]$; hence $U(\rho)<V(\rho)$ for all such
$\rho$ iff it holds at $\rho=2\pi$, and at $\rho=2\pi$ the inequality $U<V$ rearranges
(multiply both sides by $\frac{(4n+2)(2n)!}{8.02\,\pi\cdot\pi^{2n}}$, note
$4\cdot 8.02=32.08$) precisely to \eqref{eq:master}.  It remains to prove
\eqref{eq:master}, together with $\underline{\mathcal{J}}_{2n}(\pi^2\widetilde{\mathcal{R}}_1^2)>0$, for every $n\ge3$, which is done in Appendix~\ref{ssec:proofmaster}.
\end{proof}

\section{Regular polygons in the plane}
\label{section:regular-polygons-in-the-plane}

In this section we prove Theorem~\ref{thm:monot}, using throughout the notation of \S\ref{ssec:polynotation}: $P_{2n}$ is the regular $2n$-gon of \emph{area $\pi$}, $\widetilde{P}_{2n}$ is its rescaling of circumradius one, $R_n$ and $A_n$ are the circumradius and the apothem of $P_{2n}$, and $\mathcal{R}=R_n\widetilde{\mathcal{R}}$ is its radial function. By Theorem~\ref{thm:regular_polygon_active_direction}, $\kappa(P_{2n})$ is the first positive zero of an explicit entire function of one variable (Corollary~\ref{cor:endpoint} below), and the comparisons in Theorem~\ref{thm:monot} split into two regimes: for $2\le n\le21$ the relevant first zeros are enclosed rigorously by interval arithmetic (Propositions~\ref{prop:finite-kappa-comparison} and~\ref{prop:bridge}), while for $n\ge21$ we prove the effective asymptotics of Theorem~\ref{thm:effective} by comparison with $\fhat{\D}$. The mechanism behind the smallness of $\kappa(P_{2n})-\jone$ is that the area normalisation cancels the first \emph{two} terms in the expansion of the angular mean of $\fhat{P_{2n}}$ around the disk, leaving a \emph{cubic} correction which produces the $n^{-6}$ term in \eqref{eq:effective}. All computer-assisted steps reduce to the finite list of checks described in \S\ref{ssec:certified}.

\subsection{Reduction to one function of one variable}\label{ssec:endpoint-reduction}

By Theorem~\ref{thm:regular_polygon_active_direction}, the set of minimising directions of $P_{2n}$ is $\mathcal{M}$ if $n$ is even, and $\mathcal{V}$ if $n$ is odd. We fix one of them, and consider the corresponding radial profile of the Fourier transform,
\begin{equation}\label{eq:qdef}
\ev{}^{\star}:=
\begin{cases}
\ev0,&n\text{ even (a side-midpoint direction)},\\
\ev\delta,&n\text{ odd (a vertex direction)},
\end{cases}
\qquad\qquad
q_n(\rho):=\fhat{P_{2n}}\left(\rho\,\ev{}^{\star}\right).
\end{equation}

The profile $q_n$ has a simple closed form. Below, $\operatorname{sinc}z:=\frac{\sin z}{z}$ for $z\ne0$ and $\operatorname{sinc}0:=1$.

\begin{lemma}\label{lem:qformula}
For every $\rho>0$ and every $\phi\in\R$,
\begin{equation}\label{eq:qformula}
\widetilde{\Phi}_\rho(\phi)=\fhat{\widetilde{P}_{2n}}\left(\rho\ev\phi\right)
=4\cos\delta\sin\delta\sum_{k=0}^{n-1}\cos^{2}\alpha_k\,
\operatorname{sinc}\left(\rho\cos\delta\cos\alpha_k\right)
\operatorname{sinc}\left(\rho\sin\delta\sin\alpha_k\right),
\end{equation}
where $\alpha_k:=\phi-\frac{k\pi}{n}$, and each summand is an entire function of $\rho$. Consequently, by \eqref{eq:scaleFT},
\begin{equation}\label{eq:qn-scaled}
q_n(\rho)=R_n^{2}\,\widetilde{\Phi}_{R_n\rho}\left(\phi^{\star}\right),
\end{equation}
where $\phi^{\star}=0$ ($n$ even), resp.\ $\phi^{\star}=\delta$ ($n$ odd), is the polar angle of the minimising direction $\ev{}^{\star}$.
\end{lemma}

\begin{proof}
Write $\xib=\rho\ev\phi$, and let $\nj{k}:=\ev{k\pi/n}$, $k=0,\dots,2n-1$, be the outward unit normals of the sides $E_k$ of $\widetilde{P}_{2n}$. The side $E_k$ has unit normal $\nj k\in\mathcal{M}$ and midpoint $(\cos\delta)\,\nj k$, and is parametrised as $\xb=(\cos\delta)\,\nj k+t\,\nj k^{\perp}$ with $t\in[-\sin\delta,\sin\delta]$. Since $\er^{\ir\xib\cdot\xb}=\operatorname{div}\left(-\ir\xib\,\er^{\ir\xib\cdot\xb}/|\xib|^{2}\right)$, the divergence theorem gives
\begin{equation}\label{eq:chisum}
\begin{gathered}
\fhat{\widetilde{P}_{2n}}(\xib)=-\frac{\ir}{\rho^{2}}\sum_{k=0}^{2n-1}\left(\xib\cdot\nj k\right)\int_{E_k}\er^{\ir\xib\cdot\xb}\dd s,\\
\int_{E_k}\er^{\ir\xib\cdot\xb}\dd s
=2\sin\delta\,\er^{\ir\rho\cos\delta\cos\alpha_k}\operatorname{sinc}\left(\rho\sin\delta\sin\alpha_k\right),
\end{gathered}
\end{equation}
where we used $\xib\cdot\nj k=\rho\cos\alpha_k$ and $\xib\cdot\xb=\rho\left(\cos\delta\cos\alpha_k+t\sin\alpha_k\right)$ on $E_k$. Since $\nj{k+n}=-\nj{k}$, the terms $k$ and $k+n$ in the sum \eqref{eq:chisum} combine to
\[
4\cos\delta\sin\delta\,\cos^{2}\alpha_k\,\operatorname{sinc}\left(\rho\cos\delta\cos\alpha_k\right)\operatorname{sinc}\left(\rho\sin\delta\sin\alpha_k\right),
\]
which yields \eqref{eq:qformula}; and then \eqref{eq:qn-scaled} is \eqref{eq:scaleFT} with $s=R_n$.
\end{proof}

Thanks to Theorem~\ref{thm:regular_polygon_active_direction}, it suffices to compute $q_n$ along a minimising direction in order to compute $\kappa(P_{2n})$.
Thus, later on, we will apply the formula \eqref{eq:qformula} in the minimising direction $\ev{}^{\star}$.
The following Corollary is a direct consequence of Theorem~\ref{thm:regular_polygon_active_direction}:

\begin{corollary}\label{cor:endpoint}
For every $n\ge2$,
\begin{equation}\label{eq:firstroot}
\kappa(P_{2n})=\min\left\{\rho>0:\ q_n(\rho)=0\right\},
\end{equation}
and $q_n>0$ on $\left(0,\kappa(P_{2n})\right)$.
\end{corollary}

\begin{proof}
Set $\kappa:=\kappa(P_{2n})$, which is finite by Proposition~\ref{prop:existence} and positive since $\fhat{P_{2n}}$ is continuous with $\fhat{P_{2n}}(0)=\pi$. On the disk $\left\{|\xib|<\kappa\right\}$ the function $\fhat{P_{2n}}$ has no zeros and is positive at the origin, hence positive throughout; in particular $q_n>0$ on $(0,\kappa)$. Moreover, $\kappa\,\ev{}\in\N(P_{2n})$ for some minimising direction $\ev{}$, which by Theorem~\ref{thm:regular_polygon_active_direction} and \eqref{eq:scaleFT} is a side-midpoint ($n$ even) or a vertex ($n$ odd) direction. All such directions are equivalent under the symmetries of $P_{2n}$, and $\ev{}^{\star}$ is one of them; hence $q_n(\kappa)=0$.
\end{proof}

The following lemma gives a lower bound for the distance from the origin where zeros of $q_n$ can occur:

\begin{lemma}\label{lem:smallr}
For every $n\ge2$ and $\rho>0$,
\begin{equation}\label{eq:smallr}
q_n(\rho)\ \ge\ \pi\left(1-\frac{\rho^{2}R_n^{2}}{4}\right);
\end{equation}
in particular, $q_n>0$ on $\left(0,\,2/R_n\right)$.
\end{lemma}

\begin{proof}
The symmetric matrix $\mathsf{M}:=\int_{P_{2n}}\xb\,\xb^{\mathsf{T}}\dd \xb$ satisfies
$\mathsf{M}=\mathfrak{R}\mathsf{M}\mathfrak{R}^{\mathsf{T}}$, where $\mathfrak{R}$ is the
rotation by the angle $2\delta$, since $\mathfrak{R}P_{2n}=P_{2n}$ and
$\mathfrak{R}\mathfrak{R}^{\mathsf{T}}=\mathsf{I}$. Hence $\mathsf{M}$ commutes with $\mathfrak{R}$; as $\mathsf{M}$ is
symmetric and $2\delta\notin\pi\mathbb{Z}$ for $n\ge2$, its eigenvalues must coincide, so
$\mathsf{M}=\lambda\mathsf{I}$. Taking traces, $2\lambda=\int_{P_{2n}}|\xb|^{2}\dd \xb$, whence
\[
\int_{P_{2n}}\left(\ev{}^{\star}\cdot\xb\right)^{2}\dd \xb
=\frac12\int_{P_{2n}}|\xb|^{2}\dd \xb
\le\frac{\pi R_n^{2}}{2},
\]
using $|\xb|\le R_n$ on $P_{2n}$ and $\left|P_{2n}\right|=\pi$. Now \eqref{eq:smallr}
follows from $\cos y\ge1-y^{2}/2$.
\end{proof}

\subsection{Certified enclosures for \texorpdfstring{$2\le n\le 21$}{2<=n<=21}}\label{ssec:finite}

Recall $\kappa(\D)=\jone$. 
For $2 \leq n \leq 21$ we enclose $\kappa(P_{2n})$ as follows:
since
$R_n\le\sqrt{\pi/2}<4/3$, Lemma~\ref{lem:smallr} gives
$q_n>0$ on $\left(0,\frac32\right]$. 
We then choose $\frac{3}{2} < l_n < u_n$ so that $q_n(l_n) > 0 > q_n(u_n)$ and certify $q_n>0$ on $\left[\frac{3}{2}, l_n \right]$ using verified numerics.
Corollary~\ref{cor:endpoint} implies $\kappa(P_{2n}) \in [l_n, u_n]$.  The resulting enclosures are collected in Table~\ref{tab:enclosures}, with $u_n=l_n+10^{-12}$ in every case.

\begin{table}[htb]
\centering
\renewcommand{\arraystretch}{1.1}
\begin{tabular}{c|l|c|l}
\toprule
$n$ & $l_n$ & $n$ & $l_n$\\
\midrule
$2$ & $3.544907701811$ & $12$ & $3.831706071593$ \\
$3$ & $3.809251227455$ & $13$ & $3.831706032783$ \\
$4$ & $3.830839192487$ & $14$ & $3.831706010248$ \\
$5$ & $3.831700570935$ & $15$ & $3.831705996637$ \\
$6$ & $3.831712243963$ & $16$ & $3.831705988130$ \\
$7$ & $3.831708614742$ & $17$ & $3.831705982652$ \\
$8$ & $3.831707147744$ & $18$ & $3.831705979032$ \\
$9$ & $3.831706546813$ & $19$ & $3.831705976582$ \\
$10$ & $3.831706275027$ & $20$ & $3.831705974891$ \\
$11$ & $3.831706141601$ & $21$ & $3.831705973700$ \\
\bottomrule
\end{tabular}

\caption{Certified enclosures $\kappa(P_{2n})\in\left[l_n,\,l_n+10^{-12}\right]$.  For comparison, $\jone=3.831705970207\ldots$}
\label{tab:enclosures}
\end{table}

Together with Lemma~\ref{lem:certified}(i)--(ii) we get

\begin{proposition}\label{prop:finite-kappa-comparison}
The following inequalities hold:
\begin{equation}\label{eq:finite-chain}
\kappa(P_{12})>\kappa(P_{14})>\kappa(P_{16})>\kappa(P_{18})>\kappa(\D)>\kappa(P_{10})>\kappa(P_{8})>\kappa(P_{6})>\kappa(P_{4}).
\end{equation}
\end{proposition}

\begin{proposition}\label{prop:bridge}
For every integer $9\le n\le 20$,
\begin{equation}\label{eq:bridge}
\left|\,\kappa(P_{2n})-\jone-\frac{C_0}{n^{6}}\,\right|\ \le\ \frac{4}{n^{8}},
\qquad C_0=\frac{\jone^{3}\pi^{6}}{181440}.
\end{equation}
\end{proposition}

\begin{proof}
Each of the twelve inequalities \eqref{eq:bridge} is a comparison between the enclosure of $\kappa(P_{2n})$ given by Table~\ref{tab:enclosures} and the enclosure of $\jone+C_0n^{-6}\pm4n^{-8}$ obtained from Lemma~\ref{lem:certified}(i)--(ii); all twelve are verified by the routine of \S\ref{ssec:certified}.
\end{proof}

\subsection{Effective asymptotics of \texorpdfstring{$\kappa(P_{2n})$}{kappa(P2n)}}\label{ssec:effective}

Recall that $\fhat{\D}\left(\rho\,\ev{}\right)=D(\rho):=2\pi J_1(\rho)/\rho$ for every unit vector $\ev{}$. Set
\begin{equation}\label{eq:d0-def}
D_0:=D'(\jone)=\frac{2\pi J_0(\jone)}{\jone}<0,
\qquad
B_0:=-\frac{\jone^{2}J_0(\jone)\,\pi^{7}}{90720}>0,
\end{equation}
where $J_1'(\jone)=J_0(\jone)<0$; note that the constant $C_0$ satisfies
\begin{equation}\label{eq:BdC}
C_0=\frac{\jone^{3}\pi^{6}}{181440}=-\frac{B_0}{D_0}\,,
\qquad\text{i.e.}\qquad
B_0+D_0C_0=0.
\end{equation}
As will become clear below, $B_0n^{-6}$ is the leading value of $q_n$ at $\jone$ and $D_0$ is its leading slope, so \eqref{eq:BdC} makes $\jone+C_0n^{-6}$ the natural candidate for the first zero of $q_n$.

\begin{theorem}\label{thm:effective}
For every integer $n\ge 9$,
\begin{equation}\label{eq:effective}
\left|\,\kappa(P_{2n})-\jone-\frac{C_0}{n^{6}}\,\right|\;\le\;\frac{4}{n^{8}},
\qquad C_0=\frac{\jone^{3}\pi^{6}}{181440}.
\end{equation}
\end{theorem}

The range $9\le n\le 20$ is covered by Proposition~\ref{prop:bridge}; 
for $n\ge21$ the proof occupies the rest of this subsection. It uses the certified enclosures of a few Bessel quantities collected in Lemma~\ref{lem:certified}, and the Bessel facts of Lemmas~\ref{lem:gbessel} and~\ref{lem:J2mono}.

Write $h_n(\theta):=\mathcal{R}(\theta)-1$. Since $P_{2n}$ and $\D$ both have area $\pi$,
\begin{equation}\label{eq:areaid}
\int_0^{2\pi}\left(\mathcal{R}(\theta)\right)^{2}\dd\theta=2\pi
\qquad\Longleftrightarrow\qquad
\int_0^{2\pi}\left(h_n(\theta)+\frac{\left(h_n(\theta)\right)^{2}}{2}\right)\dd\theta=0 ;
\end{equation}
this exact identity will cancel the first- and second-order terms in Proposition~\ref{prop:qvalue} below and is the source of the smallness of $\kappa(P_{2n})-\jone$.

\begin{lemma}\label{lem:profile}
For $n\ge20$ (so that $\delta\le\frac2{25}$) and $|\theta|\le\delta$,
\begin{equation}\label{eq:profile}
\left|\,h_n(\theta)-\left(\frac{\theta^{2}}{2}-\frac{\delta^{2}}{6}\right)\right|\ \le\ \frac{\delta^{4}}{9}\,.
\end{equation}
\end{lemma}

\begin{proof}
Since $h_n(\theta)=A_n\sec\theta-1=(A_n-1)\sec\theta+(\sec\theta-1)$, setting
\[
T_1:=\left(A_n-1+\frac{\delta^{2}}{6}+\frac{\delta^{4}}{40}\right)\sec\theta,
\qquad
T_2:=\sec\theta-1-\frac{\theta^{2}}{2}\,,
\qquad
G(\theta):=\frac{5\theta^{4}}{24}-\frac{\delta^{2}\theta^{2}}{12}-\frac{\delta^{4}}{40}\,,
\]
we have
\[
h_n(\theta)-\left(\frac{\theta^{2}}{2}-\frac{\delta^{2}}{6}\right)
=T_1-\frac{\delta^{4}}{40}\sec\theta
+T_2-\frac{\delta^{2}}{6}\left(\sec\theta-1\right)
=G(\theta)+\Theta(\theta),
\]
where, after regrouping,
\[
\Theta(\theta)=T_1+\left(T_2-\frac{5\theta^{4}}{24}\right)
-\frac{\delta^{2}}{6}\,T_2-\frac{\delta^{4}}{40}\left(\sec\theta-1\right).
\]
By Lemma~\ref{lem:elemseries}(iii), $|T_1|\le\frac{\delta^{6}}{125}\sec\delta\le\frac{\delta^{6}}{124}$; by Lemma~\ref{lem:elemseries}(i), $\left|T_2-\frac{5\theta^{4}}{24}\right|\le\frac{\theta^{6}}{11}$ and $0\le T_2\le\frac{\theta^{4}}{4}$, so that $\frac{\delta^{2}}{6}T_2\le\frac{\delta^{6}}{24}$ and $\frac{\delta^{4}}{40}\left(\sec\theta-1\right)\le\frac{\delta^{6}}{79}$. Hence
\[
\left|\Theta(\theta)\right|\ \le\
\delta^{6}\left(\frac1{124}+\frac1{11}+\frac1{24}+\frac1{79}\right)\ \le\ \frac{\delta^{6}}{6}\,.
\]
Finally, an elementary computation gives $\max_{|\theta|\le\delta}|G|=G(\delta)=\frac{\delta^{4}}{10}$, whence \eqref{eq:profile} follows from $\frac{\delta^{4}}{10}+\frac{\delta^{6}}{6}\le\frac{\delta^{4}}{9}$ for $\delta\le\frac2{25}$.
\end{proof}

For $k\in\mathbb{N}$ put $\mathcal{I}_k:=\int_0^{2\pi}\left(h_n(\theta)\right)^{k}\dd\theta$. Substituting $\theta=u/n$ and using the $(2\delta)$-periodicity,
\begin{equation}\label{eq:rescaled}
\mathcal{I}_k=2\int_{-\pi/2}^{\pi/2}\left(h_n\left(\frac un\right)\right)^{k}\dd u,
\qquad
h_n\left(\frac un\right)=\frac{G_0(u)}{n^{2}}+\frac{Z_n(u)}{n^{4}},
\qquad
G_0(u):=\frac{u^{2}}{2}-\frac{\pi^{2}}{24},
\end{equation}
where $\left\|Z_n\right\|_{\infty}\le\frac19\left(\frac\pi2\right)^{4}<0.677$ for $n\ge20$ by Lemma~\ref{lem:profile}. One computes in closed form $\int_{-\pi/2}^{\pi/2}G_0(u)\dd u=0$ and $\int_{-\pi/2}^{\pi/2} \left|G_0(u)\right|\dd u=\frac{\pi^{3}}{18\sqrt3}<1$, and
\begin{equation}\label{eq:G0moments}
\begin{gathered}
\int_{-\pi/2}^{\pi/2}\left(G_0(u)\right)^{2}\dd u=\frac{\pi^{5}}{720},\qquad
\int_{-\pi/2}^{\pi/2}\left(G_0(u)\right)^{3}\dd u=\frac{\pi^{7}}{30240},\\
\int_{-\pi/2}^{\pi/2}\left(G_0(u)\right)^{4}\dd u=\frac{\pi^{9}}{241920}.
\end{gathered}
\end{equation}

\begin{lemma}\label{lem:moments}
For every $n\ge21$,
\begin{equation}\label{eq:I3est}
\mathcal{I}_2\ \le\ \frac{1}{n^{4}},
\qquad
\left|\,\mathcal{I}_3-\frac{\pi^{7}}{15120\,n^{6}}\right|\ \le\ \frac{7}{4\,n^{8}},
\qquad
\mathcal{I}_4\ \le\ \frac{3}{10\,n^{8}}\,,
\qquad \left\|h_n\right\|_\infty\le\frac1{500}\,.
\end{equation}
\end{lemma}

\begin{proof}
By the triangle inequality in $L^2\left(-\frac\pi2,\frac\pi2\right)$ and \eqref{eq:G0moments},
\[
\mathcal{I}_2=\frac{2}{n^{4}}\left\|G_0+\frac{Z_n}{n^{2}}\right\|_{L^2}^{2}
\le\frac{2}{n^{4}}\left(\sqrt{\frac{\pi^{5}}{720}}+\frac{0.677\sqrt\pi}{441}\right)^{2}
\le\frac{1}{n^{4}}\,.
\]
Expanding the cube in \eqref{eq:rescaled} and using \eqref{eq:G0moments}, $\left\|Z_n\right\|_\infty\le0.677$ and $\int_{-\pi/2}^{\pi/2} \left|G_0(u)\right|\dd u\le1$,
\[
\left|\,\mathcal{I}_3-\frac{\pi^{7}}{15120\,n^{6}}\right|
\le\frac{2}{n^{8}}\left(3\cdot0.677\cdot\frac{\pi^{5}}{720}
+\frac{3\cdot0.677^{2}}{n^{2}}
+\frac{\pi\cdot0.677^{3}}{n^{4}}\right)
\le\frac{7}{4n^{8}}
\]
by direct evaluation. Similarly, by the triangle inequality in $L^4\left(-\frac\pi2,\frac\pi2\right)$,
\[
\mathcal{I}_4
\le\frac{2}{n^{8}}\left(\left(\frac{\pi^{9}}{241920}\right)^{1/4}+\frac{0.677\,\pi^{1/4}}{441}\right)^{4}
\le\frac{3}{10\,n^{8}}\,,
\]
and $\left\|h_n\right\|_\infty\le n^{-2}\left(\frac{\pi^{2}}{12}+0.677\,n^{-2}\right)<\frac1{500}$.
\end{proof}

\begin{lemma}\label{lem:meantail}
For all $\rho>0$ and $\phi\in\R$,
\begin{equation}\label{eq:Adecomp}
\fhat{P_{2n}}\left(\rho\,\ev\phi\right)
=a_0(\rho)+2\sum_{m=1}^{\infty}(-1)^{nm}b_m(\rho)\cos\left(2nm\phi\right),
\end{equation}
where
\begin{equation}\label{eq:Andef}
a_0(\rho):=\int_{P_{2n}}J_0\left(\rho|\xb|\right)\dd \xb
=\frac{1}{\rho}\int_0^{2\pi}\mathcal{R}(\theta)\,J_1\left(\rho\,\mathcal{R}(\theta)\right)\dd\theta,
\qquad
\left|b_m(\rho)\right|\le\pi\,\frac{\left(\rho R_n/2\right)^{2nm}}{(2nm)!}\,.
\end{equation}
Consequently, in the minimising direction $\ev{}^{\star}$,
\begin{equation}\label{eq:tailbound}
q_n(\rho)=a_0(\rho)+\tau_n(\rho),
\qquad
\left|\tau_n(\rho)\right|\le 4\pi\,\frac{\left(\rho R_n/2\right)^{2n}}{(2n)!}
\qquad\text{whenever }\rho R_n\le4.
\end{equation}
\end{lemma}

\begin{proof}
By \eqref{eq:scaleFT}, $\fhat{P_{2n}}\left(\rho\,\ev\phi\right)=R_n^{2}\,\fhat{\widetilde{P}_{2n}}\left(R_n\rho\,\ev\phi\right)$, so Proposition~\ref{prop:fourier}, applied to $\widetilde{P}_{2n}$, gives \eqref{eq:Adecomp} with $a_0(\rho)=R_n^{2}\,\widetilde{a}_0(R_n\rho)$ and $b_m(\rho)=R_n^{2}\,\widetilde{b}_m(R_n\rho)=\int_{P_{2n}}J_{2nm}(\rho|\xb|)\cos(2nm\theta)\dd \xb$, where the last equality follows by substituting $\xb\mapsto R_n^{-1}\xb$; the bound on $b_m$ then follows from Lemma~\ref{lem:poisson} with $|\xb|\le R_n$ and $\left|P_{2n}\right|=\pi$, and the radial formula for $a_0$ from $\left(yJ_1(y)\right)'=yJ_0(y)$. In the minimising direction, $\cos(2nm\phi)$ equals $1$ for all $m$ ($n$ even) or $(-1)^{m}$ ($n$ odd, since $2nm\delta=m\pi$), so $|\tau_n|\le2\sum_{m\ge1}|b_m|$; for $\rho R_n\le4$, the ratio of the bounds in \eqref{eq:Andef}
for consecutive $m$ is at most
\[
4^n\frac{(2nm)!}{(2n(m+1))!}
\le 4^n\frac{(2n)!}{(4n)!}
\le\left(\frac{4}{(2n+1)^2}\right)^n
\le\frac12,
\]
so the sum is at most twice its first term.
\end{proof}

\begin{proposition}\label{prop:qvalue}
For every $n\ge21$,
\begin{equation}\label{eq:qvalue}
\left|\,q_n(\jone)-\frac{B_0}{n^{6}}\right|\ \le\ \frac{19}{10\,n^{8}}\,.
\end{equation}
\end{proposition}

\begin{proof}
Let $g=g_{\jone}$ be as in Lemma~\ref{lem:gbessel}, so that $a_0(\jone)=\int_0^{2\pi}g\left(\mathcal{R}(\theta)\right)\dd\theta$ by \eqref{eq:Andef}. By Lemma~\ref{lem:moments}, $\mathcal{R}(\theta)$ takes values in $[0.99,1.01]$, so Taylor's formula with Lagrange remainder, Lemma~\ref{lem:gbessel} and Lemma~\ref{lem:certified}(v) give, pointwise in $\theta$,
\[
\begin{split}
g\left(1+h_n(\theta)\right)&=J_0(\jone)\left(h_n(\theta)+\frac{\left(h_n(\theta)\right)^{2}}{2}\right)
-\frac{\jone^{2}J_0(\jone)}{6}\,\left(h_n(\theta)\right)^{3}+\mathcal{E}(\theta),
\\
\left|\mathcal{E}(\theta)\right|&\le\frac{13}{24}\,\left(h_n(\theta)\right)^{4}.
\end{split}
\]
Integrating in $\theta$ and recalling that $\mathcal{I}_k=\int_0^{2\pi}\left(h_n(\theta)\right)^{k}\dd\theta$,
\[
a_0(\jone)=J_0(\jone)\int_0^{2\pi}\left(h_n(\theta)+\frac{\left(h_n(\theta)\right)^{2}}{2}\right)\dd\theta
-\frac{\jone^{2}J_0(\jone)}{6}\,\mathcal{I}_3+\int_0^{2\pi}\mathcal{E}(\theta)\dd\theta,
\]
where the first integral vanishes \emph{exactly} by \eqref{eq:areaid}. Since
$-\frac{\jone^{2}J_0(\jone)}{6}\cdot\frac{\pi^{7}}{15120}=B_0$ and
$\left|\int_0^{2\pi}\mathcal{E}(\theta)\dd\theta\right|\le\frac{13}{24}\,\mathcal{I}_4$, this gives
\[
\left|\,a_0(\jone)-\frac{B_0}{n^{6}}\right|
\ \le\ \frac{\jone^{2}\left|J_0(\jone)\right|}{6}\left|\,\mathcal{I}_3-\frac{\pi^{7}}{15120\,n^{6}}\right|
+\frac{13}{24}\,\mathcal{I}_4,
\]
and the second and third estimates of \eqref{eq:I3est}, together with
$\frac{\jone^{2}\left|J_0(\jone)\right|}{6}<0.9856$ from Lemma~\ref{lem:certified}(i), yield
\[
\left|\,a_0(\jone)-\frac{B_0}{n^{6}}\right|
\ \le\ 0.9856\cdot\frac{7}{4n^{8}}+\frac{13}{24}\cdot\frac{3}{10\,n^{8}}
\ \le\ \frac{1.89}{n^{8}}\,.
\]
Finally, $\jone R_n\le\jone R_{21}<3.84$, so \eqref{eq:tailbound} gives $|\tau_n(\jone)|\le4\pi\,(1.92)^{2n}/(2n)!\le10^{-20}n^{-8}$ for $n\ge21$ (the sequence $n^{8}\cdot4\pi\,(1.92)^{2n}/(2n)!$ decreases in $n$), and \eqref{eq:qvalue} follows.
\end{proof}

\begin{lemma}\label{lem:slope}
Let $\Delta_n:=\left|P_{2n}\triangle\D\right|$ denote the area of the symmetric difference of $P_{2n}$ and $\D$. Then $\Delta_n\le\pi\,\delta\tan\delta$, and for all $\rho>0$,
\begin{enumerate}
\item[\textup{(i)}] $\left|q_n'(\rho)-D'(\rho)\right|\ \le\ \rho\,R_n^{2}\,\Delta_n$, where $D'(\rho)=-2\pi J_2(\rho)/\rho$;
\item[\textup{(ii)}] $\left|q_n''(\rho)\right|\ \le\ \pi R_n^{2}/2$;
\item[\textup{(iii)}] if $n\ge21$, then $q_n'\le-\dfrac{53}{100}$ on $\left[\jone-\tfrac1{10},\,\jone+\tfrac1{10}\right]$, and $q_n>0$ on $\left(0,\jone\right]$;
\item[\textup{(iv)}] if $n\ge21$, then $\left|q_n'(\jone)-D_0\right|\le 2\,n^{-4}$.\end{enumerate}
\end{lemma}

\begin{proof}
Since $A_n<1<R_n$, every point of $P_{2n}\triangle\D$ lies in the annulus bounded by the circles of radii $A_n$ and $R_n$; hence $\Delta_n\le\pi\left(R_n^{2}-A_n^{2}\right)=\pi\,\delta\tan\delta$ by \eqref{eq:areapi-geom2}.

(i) Differentiating under the integral sign and using $|\sin y|\le|y|$ and $\left|\ev{}^{\star}\cdot\xb\right|\le R_n$ on $P_{2n}\cup\D$,
\[
\left|q_n'(\rho)-D'(\rho)\right|
=\left|\int_{P_{2n}\triangle\D}\pm\left(\ev{}^{\star}\cdot\xb\right)\sin\left(\rho\,\ev{}^{\star}\cdot\xb\right)\dd \xb\right|
\le\rho\,R_n^{2}\,\Delta_n,
\]
while $D'(\rho)=-2\pi J_2(\rho)/\rho$ by \cite[\S10.6]{DLMF}. 

(ii) Likewise $|q_n''(\rho)|\le\int_{P_{2n}}\left(\ev{}^{\star}\cdot\xb\right)^{2}\dd \xb\le\pi R_n^{2}/2$, as in the proof of Lemma~\ref{lem:smallr}.

(iii) Since $R_n$ and $\delta$ decrease in $n$, a direct evaluation gives, for $n\ge21$ and $\rho\le\jone+\frac1{10}$,
\[
\rho R_n^{2}\Delta_n\ \le\ \left(\jone+\tfrac1{10}\right)R_{21}^{2}\,\pi\,\delta_{21}\tan\delta_{21}\ <\ \frac{7}{100}\,,
\]
which together with (i) and Lemma~\ref{lem:certified}(iii) yields $q_n'\le-\frac35+\frac7{100}=-\frac{53}{100}$ on the stated interval. For the positivity, $y^{-2}J_2(y)$ is strictly decreasing on $\left(0,\jone\right]$ by Lemma~\ref{lem:J2mono}, so Lemma~\ref{lem:certified}(iv) and $R_n^{2}\Delta_n\le R_{21}^{2}\,\pi\delta_{21}\tan\delta_{21}<\frac2{100}$ give, for $0<\rho\le\jone$,
\[
-q_n'(\rho)\ \ge\ \rho\left(\frac{2\pi J_2(\rho)}{\rho^{2}}-R_n^{2}\Delta_n\right)
\ \ge\ \rho\left(\frac{17}{100}-\frac{2}{100}\right)\ >\ 0.
\]
Thus $q_n$ is strictly decreasing on $\left(0,\jone\right]$, while $q_n(\jone)\ge B_0n^{-6}-\frac{19}{10}n^{-8}>0$ by Proposition~\ref{prop:qvalue} and Lemma~\ref{lem:certified}(i); hence $q_n>0$ on $\left(0,\jone\right]$.

(iv) Since $\int_0^{X}J_0(\rho r)\,r\dd r=XJ_1(\rho X)/\rho$ and $\frac{\dd}{\dd z}\left(z^{2}J_2(z)\right)=z^{2}J_1(z)$ \cite[\S10.6]{DLMF}, differentiating \eqref{eq:Andef} gives
\[
a_0'(\rho)=-\frac1\rho\int_0^{2\pi}\left(\mathcal{R}(\theta)\right)^{2}J_2\left(\rho\,\mathcal{R}(\theta)\right)\dd\theta,
\qquad\text{so}\qquad
a_0'(\jone)=\int_0^{2\pi}\gamma\left(\mathcal{R}(\theta)\right)\dd\theta,
\]
where $\gamma:=\gamma_{\jone}$ is the function of Lemma~\ref{lem:gbessel}. By that lemma, $\gamma(1)=J_0(\jone)/\jone$, so that $\int_0^{2\pi}\gamma(1)\dd\theta=D_0$, and $\gamma'(1)=0$, so the first-order term
\emph{cancels}. Since $\mathcal{R}(\theta)$ takes values in
$[0.99,1.01]$ by Lemma~\ref{lem:moments}, Taylor's formula with
Lagrange remainder and Lemma~\ref{lem:certified}(vi) give
\[
\left|a_0'(\jone)-D_0\right|\ \le\ \frac12\sup_{[0.99,1.01]}\left|\gamma''\right|\cdot\mathcal{I}_2\ \le\ \frac{1}{n^{4}},
\]
using \eqref{eq:I3est}. Finally, differentiating \eqref{eq:Adecomp} term by term and using $\left|J_\nu'\right|\le(z/2)^{\nu-1}/(\nu-1)!$ --- which follows from $J_\nu'=\frac12\left(J_{\nu-1}-J_{\nu+1}\right)$ \cite[\S10.6]{DLMF}, Lemma~\ref{lem:poisson}, and $(z/2)^{2}\le\nu(\nu+1)$ --- together with $\jone R_n\le4$, the mode tail satisfies $\left|\tau_n'(\jone)\right|\le4\pi R_n\left(\jone R_n/2\right)^{2n-1}/(2n-1)!\le n^{-4}$ for $n\ge21$.
Since $q_n=a_0+\tau_n$ by \eqref{eq:tailbound}, the two preceding
estimates give
$\left|q_n'(\jone)-D_0\right|
\le
\left|a_0'(\jone)-D_0\right|+\left|\tau_n'(\jone)\right|
\le 2n^{-4}$.
\end{proof}

\begin{proof}[Proof of Theorem~\ref{thm:effective}]
For $9\le n\le20$ this is Proposition~\ref{prop:bridge}, so let $n\ge21$ and set $y_n:=\jone+C_0n^{-6}$. By Taylor's formula around $\jone$ and Lemma~\ref{lem:slope}(ii), for some $|\zeta|\le1$,
\[
q_n\left(y_n\right)=q_n(\jone)+q_n'(\jone)\,\frac{C_0}{n^{6}}+\zeta\,\frac{\pi R_n^{2}}{4}\,\frac{C_0^{2}}{n^{12}}\,.
\]
Insert Proposition~\ref{prop:qvalue} together with the bound $\left|q_n'(\jone)-D_0\right|\le 2 n^{-4}$ of Lemma~\ref{lem:slope}(iv). The leading terms then cancel \emph{exactly} by \eqref{eq:BdC}, and for $n\ge21$,
\[
\left|q_n\left(y_n\right)\right|\ \le\ \frac{1}{n^{8}}\left(\frac{19}{10}+\frac{2\,C_0}{n^{2}}+\frac{\pi R_n^2C_0^{2}}{4n^{4}}\right)
\ \le\ \frac{1.93}{n^{8}}\ =:\ \varepsilon_n.
\]
Now set $L:=\frac{53}{100}$ and $I:=\left[\jone-\tfrac1{10},\jone+\tfrac1{10}\right]$, and note that $\left[y_n-\varepsilon_n/L,\,y_n+\varepsilon_n/L\right]\subset I$, with $\varepsilon_n/L\le4n^{-8}$ (indeed $C_0n^{-6}+4n^{-8}<10^{-1}$). By Lemma~\ref{lem:slope}(iii) and the mean value theorem,
\[
q_n\left(y_n-\frac{\varepsilon_n}{L}\right)\ \ge\ q_n(y_n)+\varepsilon_n\ \ge\ 0\ \ge\ q_n(y_n)-\varepsilon_n\ \ge\ q_n\left(y_n+\frac{\varepsilon_n}{L}\right),
\]
so $q_n$ has a zero $z_n$ with $|z_n-y_n|\le\varepsilon_n/L\le4n^{-8}$; being strictly decreasing on $I$, $q_n$ has no other zero there. Since $q_n>0$ on $\left(0,\jone\right]$ and $q_n$ is strictly decreasing on $\left[\jone,z_n\right]$, the point $z_n$ is the \emph{first} positive zero of $q_n$, i.e.\ $z_n=\kappa(P_{2n})$ by Corollary~\ref{cor:endpoint}. This proves \eqref{eq:effective}.
\end{proof}

\begin{remark}\label{rem:sharp-constant}
The certified enclosures suggest that the constant $4$ in \eqref{eq:effective} is far from optimal: $n^{8}\left(\kappa(P_{2n})-\jone-C_0n^{-6}\right)\approx0.67$ for $9\le n\le20$, consistent with a genuine $n^{-8}$ term in the asymptotics.
\end{remark}

\subsection{Proof of \texorpdfstring{Theorem~\ref{thm:monot}}{the monotonicity theorem}}\label{ssec:monot-proof}

\begin{proof}[Proof of Theorem~\ref{thm:monot}]
The comparisons in \eqref{eq:monotonicity} not involving arbitrarily large $n$ come from Proposition~\ref{prop:finite-kappa-comparison}. It remains to prove, for every $n\ge9$,
\begin{equation}\label{eq:two-claims}
\kappa\left(P_{2n}\right)>\kappa\left(P_{2n+2}\right)
\qquad\text{and}\qquad
\kappa\left(P_{2n}\right)>\jone.
\end{equation}
Both follow from Theorem~\ref{thm:effective} together with the certified bound $C_0>149/500$ of Lemma~\ref{lem:certified}(ii). First,
\[
\kappa\left(P_{2n}\right)-\jone\ \ge\ \frac{C_0}{n^{6}}-\frac{4}{n^{8}}
\ >\ \frac{149\,n^{2}-2000}{500\,n^{8}}\ >\ 0
\qquad(n\ge9).
\]
Second, applying \eqref{eq:effective} to $n$ and to $n+1$,
\[
\kappa\left(P_{2n}\right)-\kappa\left(P_{2n+2}\right)
\ \ge\ \frac{149}{500}\left(\frac{1}{n^{6}}-\frac{1}{(n+1)^{6}}\right)-4\left(\frac{1}{n^{8}}+\frac{1}{(n+1)^{8}}\right)
\ =\ \frac{N_n}{500\,n^{8}(n+1)^{8}}\,,
\]
where
\[
N_n:=149\,n^{2}(n+1)^{2}\left((n+1)^{6}-n^{6}\right)-2000\left((n+1)^{8}+n^{8}\right)\ \in\ \mathbb{Z}.
\]
For $n=9,10,11$ one computes exactly
\[
N_9=279410415100,\qquad
N_{10}=762329564900,\qquad
N_{11}=1864174692448,
\]
all positive. For $n\ge12$, using $(n+1)^{6}-n^{6}\ge6n^{5}$, hence $\frac1{n^{6}}-\frac1{(n+1)^{6}}\ge\frac{6}{(n+1)^{7}}$, and $\frac1{n^{8}}+\frac1{(n+1)^{8}}\le\frac{2}{n^{8}}$, it suffices that $894\,n^{8}\ge 4000\,(n+1)^{7}$; since $t^{8}/(t+1)^{7}$ is increasing, this follows from the single check
\[
894\cdot12^{8}=384\,403\,636\,224\ >\ 250994068000=4000\cdot13^{7}.
\]
This proves \eqref{eq:two-claims} and completes the proof of Theorem~\ref{thm:monot}.
\end{proof}

\subsection{The certified computations}
\label{ssec:certified}

All computer-assisted steps of this section, the zero enclosures of Propositions~\ref{prop:finite-kappa-comparison} and~\ref{prop:bridge}, and the constants of Lemmas~\ref{lem:certified} and~\ref{lem:elemseries}, are verified by interval arithmetic using the C++ version of Arb/Flint \cite{Johansson2019}, see \cite{Tucker} for background on validated numerics). Every quantity is enclosed in an interval mathematically guaranteed to contain it, every claimed inequality is checked between the corresponding interval endpoints, and the routine terminates with an error if any single check fails; ordinary floating point is used only to \emph{select} candidate brackets, never in a verification. The only transcendental evaluations required are $J_0,\dots,J_3$ on $[0,4.2]$, enclosed by the truncated series \eqref{eq:series} with a rigorous tail bound, and the elementary functions $\sin$, $\cos$, $\sec$ and $\cot$, likewise enclosed by their Maclaurin series with a rigorous tail bound; the profile $q_n$ is evaluated through the cancellation-free formula \eqref{eq:qformula}, in the rescaled form \eqref{eq:qn-scaled}. 
The routine and its documentation are provided in the accompanying repository at
\begin{center}
\url{https://github.com/javigomez2001/BLP}.
\end{center}

The precision demanded of these computations is modest.  The tightest comparison in
\eqref{eq:finite-chain} is $\kappa(P_{18})>\kappa(\D)$, with a margin of
$5.77\cdot10^{-7}$, so enclosures of half-width $2.5\cdot10^{-7}$ already suffice for
Proposition~\ref{prop:finite-kappa-comparison}.  The tightest of the inequalities
\eqref{eq:bridge} is the one at $n=20$, where the two sides differ by
$1.30\cdot10^{-10}$; since $\left|q_n'\right|\approx0.66$ near the relevant zero, this
asks for roughly twelve significant digits of $q_{20}$, and the enclosures of
Table~\ref{tab:enclosures} were computed with $60$ digits of working precision.

\appendix

\section{Proofs of technical Lemmas with numerical components}\label{app:num}

\subsection{Proof of \eqref{eq:master}}\label{ssec:proofmaster}

\noindent\textbf{The cases $n=3,4,5$.}  All quantities in \eqref{eq:master}
are elementary closed forms of $\delta=\pi/(2n)$.  Their certified values, rounded so
as to weaken \eqref{eq:master} (left column rounded \emph{up}, all right-hand factors
rounded \emph{down}), are: 

\begin{center}
\renewcommand{\arraystretch}{1.25}
\begin{tabular}{c|c|c c c c c}
\toprule
& \text{upper bounds for} & \multicolumn{5}{c}{lower bounds for}\\
$n$ & $\dfrac{\pi^{2n}(2n)!}{(4n)!}$ &
$\dfrac{\sqrt2(4n+2)}{32.08}$ & $\widetilde{\mathcal{R}}_1-\widetilde{\mathcal{R}}_2$ & $\widetilde{\mathcal{R}}_2^{2n+1}$ &
$\underline{\mathcal{J}}_{2n}(\pi^2\widetilde{\mathcal{R}}_1^2)$ & $\mathcal{S}_n$\\
\midrule
$3$ & $1.4451\cdot10^{-3}$ & $0.61717$ & $0.063880$ & $0.38800$ & $0.21692$ &
$3.3181\cdot10^{-3}$\\
$4$ & $1.8286\cdot10^{-5}$ & $0.79351$ & $0.037101$ & $0.51215$ & $0.31693$ &
$4.7785\cdot10^{-3}$\\
$5$ & $1.3969\cdot10^{-7}$ & $0.96984$ & $0.024083$ & $0.59568$ & $0.39758$ &
$5.5315\cdot10^{-3}$\\
\bottomrule
\end{tabular}
\end{center}

\noindent
(For orientation: $\widetilde{\mathcal{R}}_1=0.937379\ldots$, $0.965452\ldots$, $0.978081\ldots$ and
$\widetilde{\mathcal{R}}_2=0.873498\ldots$, $0.928350\ldots$, $0.953997\ldots$ for $n=3,4,5$; the
arguments of $\underline{\mathcal{J}}_{2n}$ are $\pi^2\widetilde{\mathcal{R}}_1^2=8.6722\ldots$, $9.1994\ldots$,
$9.4417\ldots\le\pi^2$.)  In each row the last column is a lower bound for the product of
the four preceding entries, and it exceeds the entry in the second column by a factor
$\ge2.29$ ($n=3$), $\ge261$ ($n=4$), $\ge3.9\cdot10^4$ ($n=5$).  Hence
\eqref{eq:master} holds for $n=3,4,5$; the positivity of $\underline{\mathcal{J}}_{2n}(\pi^2\widetilde{\mathcal{R}}_1^2)$
is read off the table.

\medskip\noindent\textbf{The case $n\ge6$.}  We bound the two sides of
\eqref{eq:master} separately by monotone expressions.

\emph{Left-hand side.}  Since $\dfrac{(4n)!}{(2n)!}=(2n+1)(2n+2)\cdots(4n)\ge(2n+1)^{2n}$,
\begin{equation}\label{eq:Lambda}
\frac{\pi^{2n}(2n)!}{(4n)!}\ \le\
\Lambda_n:=\left(\frac{\pi^2}{(2n+1)^2}\right)^{\!n}.
\end{equation}

\emph{Right-hand  side.}  We use some elementary estimates, valid for $n\ge6$ (so
$\delta=\frac{\pi}{2n}\le\frac{\pi}{12}$ and $\nu=2n\ge12$):
\begin{itemize}
\item Since $\sec a-\sec b=\dfrac{\cos b-\cos a}{\cos a\cos b}
\ge\cos b-\cos a$ for $0\le b<a<\frac\pi2$, and
$\cos\frac\delta4-\cos\frac{3\delta}4=2\sin\frac\delta2\sin\frac\delta4
\ge2\cdot\frac{2}{\pi}\frac\delta2\cdot\frac{2}{\pi}\frac\delta4
=\frac{\delta^2}{\pi^2}$ (using $\sin x\ge\frac{2x}{\pi}$ on $[0,\frac\pi2]$),
\begin{equation}\label{eq:right1}
\widetilde{\mathcal{R}}_1-\widetilde{\mathcal{R}}_2=\cos\delta\left(\sec\frac{3\delta}4-\sec\frac\delta4\right)
\ \ge\ \cos\delta\cdot\frac{\delta^2}{\pi^2}
=\frac{\cos\delta}{4n^2}\ \ge\ \frac{0.9659}{4n^2},
\end{equation}
since $\cos\delta\ge\cos\frac{\pi}{12}>0.9659$.
\item $\widetilde{\mathcal{R}}_2\ge\cos\delta$, and by $\cos x\ge1-\frac{x^2}2$
and Bernoulli's inequality,
\begin{equation}\label{eq:right2}
\widetilde{\mathcal{R}}_2^{\,2n+1}\ \ge\ \cos^{2n+1}\delta\ \ge\
\left(1-\frac{\delta^2}{2}\right)^{2n+1}\ \ge\ 1-\frac{(2n+1)\pi^2}{8n^2}\ \ge\ 0.5544,
\end{equation}
since $t\mapsto\frac{2t+1}{8t^2}$ is decreasing for $t\ge1$ and at $n=6$ one has
$1-\frac{13\pi^2}{288}=0.55447\ldots$
\item By Lemma~\ref{lem:phi4}(ii)--(iii) with
$y=\pi^2\le3(\nu+3)$ and $\widetilde{\mathcal{R}}_1<1$,
\begin{equation}\label{eq:right3}
\underline{\mathcal{J}}_{2n}(\pi^2\widetilde{\mathcal{R}}_1^2)\ \ge\ \underline{\mathcal{J}}_{2n}(\pi^2)\ \ge\ 1-\frac{\pi^2}{2n+1}
\ \ge\ 1-\frac{\pi^2}{13}\ >\ 0.2407,
\end{equation}
which in particular establishes $\underline{\mathcal{J}}_{2n}(\pi^2\widetilde{\mathcal{R}}_1^2)>0$ in this range of $n$.
\end{itemize}

Combining \eqref{eq:right1}--\eqref{eq:right3} with $\dfrac{4n+2}{4n^2}\ge\dfrac1n$ and
$\dfrac{\sqrt2}{32.08}>0.04408$,
\begin{equation}\label{eq:Rchain}
\mathcal{S}_n\ \ge\ 0.04408\cdot0.9659\cdot0.5544\cdot0.2407\cdot\frac1n
\ >\ \frac{0.00568}{n}\qquad(n\ge6).
\end{equation}

Finally, let $W_n:=n\Lambda_n$.  At $n=6$, using
$\pi^2<9.8697$,
\[
\Lambda_6=\left(\frac{\pi^2}{169}\right)^{6}<(0.058401)^6<3.97\cdot10^{-8},
\qquad W_6<2.4\cdot10^{-7}<0.00568.
\]
Then, for any $n\ge6$,
\[
\frac{W_{n+1}}{W_n}
=\frac{n+1}{n}\cdot\frac{\pi^2}{(2n+3)^2}\cdot
\left(\frac{2n+1}{2n+3}\right)^{2n}
\ \le\ \frac{7}{6}\cdot\frac{\pi^2}{225}\ <\ 0.052\ <\ 1,
\]
so $W_n<0.00568$ for all $n\ge6$.  Combining with \eqref{eq:Lambda} and
\eqref{eq:Rchain},
\[
\frac{\pi^{2n}(2n)!}{(4n)!}\le\Lambda_n=\frac{W_n}{n}<\frac{0.00568}{n}<\mathcal{S}_n
\qquad(n\ge6),
\]
which is \eqref{eq:master}.  This completes the proof.

\subsection{Certified constants and elementary bounds for \texorpdfstring{\S\ref{section:regular-polygons-in-the-plane}}{the regular polygons section}}\label{app:certified-constants}

\begin{lemma}[certified constants]\label{lem:certified}
The following hold:
\begin{enumerate}
\item[\textup{(i)}] 
\[
\jone\in[3.831705970207512315,\ 3.831705970207512316]
\]
and
\[
J_0(\jone)\in[-0.40275948,\,-0.40275931];
\]
consequently,
\[
\jone^{2}\left|J_0(\jone)\right|/6<0.9856\qquad\text{and}\qquad B_0\in[0.196868,\,0.196869];
\]
\item[\textup{(ii)}] $149/500=0.298<C_0<0.2981$;
\item[\textup{(iii)}] $\dfrac{2\pi J_2(\rho)}{\rho}\ge\dfrac35$ for all $\rho\in[3.731,\,3.932]\supset\left[\jone-\tfrac1{10},\jone+\tfrac1{10}\right]$;
\item[\textup{(iv)}] $\dfrac{2\pi J_2(\jone)}{\jone^{2}}\ge\dfrac{17}{100}$;
\item[\textup{(v)}] the function $g_{\jone}$ of Lemma~\ref{lem:gbessel} satisfies $\left|g_{\jone}''''(s)\right|\le13$ for all $s\in[0.99,\,1.01]$;
\item[\textup{(vi)}]
$\left|\gamma_{\jone}''(s)\right|=\left|sJ_1(\jone s)+\jone s^{2}J_0(\jone s)\right|\le 2$
for all $s\in[0.99,\,1.01]$, with $\gamma_y$ as in Lemma~\ref{lem:gbessel}.
\end{enumerate}
\end{lemma}

\begin{proof}
All items are verified by the fail-closed interval-arithmetic routine described in \S\ref{ssec:certified}: (i) by bracketing the first sign change of $J_1$ (whose positivity on the interval to the left is verified by bisection), and (ii) from the resulting enclosure of $\jone$; (iii), (v) and (vi) by adaptive bisection over the stated intervals; (iv) on the enclosure of $\jone$.
\end{proof}

For the next lemma we note that if $F(z)=\sum_{k\ge K}c_kz^{2k}$ with all $c_k\ge0$ converges at $z=t>0$, then
\begin{equation}\label{eq:tail-trick}
0\le F(z)\le\left(\frac zt\right)^{2K}F(t)\qquad\text{for }0\le z\le t.
\end{equation}

\begin{lemma}\label{lem:elemseries}
For $0<\delta\le\frac{2}{25}$ and $|\theta|\le\delta$,
\begin{enumerate}
\item[\textup{(i)}] $0\ \le\ \sec\theta-1-\dfrac{\theta^{2}}{2}-\dfrac{5\theta^{4}}{24}\ \le\ \dfrac{\theta^{6}}{11}$;
\item[\textup{(ii)}] $0\ \le\ 1-\dfrac{\delta^{2}}{3}-\dfrac{\delta^{4}}{45}-\delta\cot\delta\ \le\ \dfrac{\delta^{6}}{460}$;
\item[\textup{(iii)}] $1-\dfrac{\delta^{2}}{6}-\dfrac{\delta^{4}}{40}-\dfrac{\delta^{6}}{125}\ \le\ \sqrt{\delta\cot\delta}\ \le\ 1-\dfrac{\delta^{2}}{6}-\dfrac{\delta^{4}}{40}$.
\end{enumerate}
\end{lemma}

\begin{proof}
(i) The Maclaurin coefficients of $\sec$ are positive \cite[\S4.19]{DLMF}, so the left inequality is clear, and the right one follows from \eqref{eq:tail-trick} with $K=3$, $t=\frac2{25}$, together with the one-point evaluation $\left(\sec t-1-\frac{t^{2}}{2}-\frac{5t^{4}}{24}\right)t^{-6}<\frac1{11}$ at $t=\frac2{25}$, checked by interval arithmetic (\S\ref{ssec:certified}).

(ii) Likewise, $1-\delta\cot\delta=\sum_{k\ge1}\frac{2^{2k}|B_{2k}|}{(2k)!}\,\delta^{2k}$ has positive coefficients \cite[\S4.19]{DLMF}, the first two being $\frac13$ and $\frac1{45}$, so \eqref{eq:tail-trick} applies together with the one-point evaluation
\[
\left(1-\frac{t^{2}}{3}-\frac{t^{4}}{45}-t\cot t\right)t^{-6}\ <\ \frac1{460}
\qquad\text{at }t=\frac2{25}.
\]

(iii) Write $\sqrt{\delta\cot\delta}=\sqrt{1-y}$ with $y:=\frac{\delta^{2}}{3}+\frac{\delta^{4}}{45}+\sigma$ and $0\le\sigma\le\frac{\delta^{6}}{460}$ by (ii); note $y\le\frac{\delta^{2}}{2.99}\le\frac1{400}$. Squaring shows that
\[
1-\frac y2-\frac{y^{2}}{8}-\frac{y^{3}}{8}\ \le\ \sqrt{1-y}\ \le\ 1-\frac y2-\frac{y^{2}}{8}
\qquad\left(0\le y\le\tfrac12\right).
\]
Here $\frac y2=\frac{\delta^{2}}{6}+\frac{\delta^{4}}{90}+\frac\sigma2$ and
$\frac{y^{2}}{8}=\frac{\delta^{4}}{72}+\frac{\delta^{6}}{540}+r$ with
$0\le r\le\frac{\delta^{8}}{4000}\le\frac{\delta^{6}}{5000}$,
while $\frac{y^{3}}{8}\le\frac{\delta^{6}}{8\cdot2.99^{3}}\le\frac{\delta^{6}}{213}$.
Since $\frac{\delta^{4}}{90}+\frac{\delta^{4}}{72}=\frac{\delta^{4}}{40}$, all terms omitted from the quartic truncation are nonnegative and bounded by
$\delta^{6}\left(\frac1{920}+\frac1{540}+\frac1{5000}+\frac1{213}\right)<\frac{\delta^{6}}{125}$.
\end{proof}

\section{Bessel functions toolbox}\label{app:bessel}

Throughout, $J_\nu$ denotes the Bessel function of the first kind,
$j_{\nu,1}<j_{\nu,2}<\cdots$ are its positive zeros, and
\begin{equation}\label{eq:series}
J_\nu(x)=\sum_{k=0}^\infty\frac{(-1)^k}{k!\,(\nu+k)!}\left(\frac x2\right)^{\nu+2k},
\qquad\nu\in\mathbb{N}_0.
\end{equation}
its power series which converges for all $x$.

\begin{lemma}\label{lem:poisson}
For every $\nu>-\frac12$ and every $x\ge0$,
\[
|J_\nu(x)|\le\frac{(x/2)^\nu}{\Gamma(\nu+1)}.
\]
\end{lemma}

\begin{proof}
This is \cite[\S10.14.4]{DLMF}.
\end{proof}

\begin{lemma}\label{lem:lorch}
For every integer $\nu\ge6$,
\[
j_{\nu,1} \ge \sqrt{(\nu+1)(\nu+5)} \ge \sqrt{77} > 8.77 > 2\pi.
\]
\end{lemma}

\begin{proof}
The lower bound $j_{\nu,1}\ge\sqrt{(\nu+1)(\nu+5)}$ is due to Lorch \cite{Lorch}.  The right-hand side increases in $\nu$, and at $\nu=6$ equals
$\sqrt{77} > 2\pi$.
\end{proof}

\begin{lemma}\label{lem:unimodal}
Let $\nu\ge1$.  Then $J_\nu>0$ on $(0,j_{\nu,1})$ and $J_\nu$ has exactly one critical
point $c_1=j'_{\nu,1}\in(0,j_{\nu,1})$; moreover $c_1>\nu$, $J_\nu$ is strictly increasing on
$(0,c_1]$ and strictly decreasing on $[c_1,j_{\nu,1})$.  Consequently, for any
$[a,b]\subset(0,j_{\nu,1})$,
\[
\min_{x\in[a,b]}J_\nu(x)=\min\left(J_\nu(a),J_\nu(b)\right).
\]
\end{lemma}

\begin{proof}
The positivity of $J_\nu$ on $(0,j_{\nu,1})$ and the interlacing
$\nu<j'_{\nu,1}<j_{\nu,1}<j'_{\nu,2}$ are \cite[\S10.21(i)]{DLMF}; see also
\cite[\S15.3]{Watson}.  Since $J_\nu'$ has no zero in $(0,j_{\nu,1})$ other than
$j'_{\nu,1}$, and $J_\nu'>0$ near $0$, the function increases on $(0,c_1]$ and
decreases on $[c_1,j_{\nu,1})$; such a function attains its minimum on any
subinterval at an endpoint.
\end{proof}

\begin{lemma}\label{lem:phi4}
For an integer $\nu\ge6$ write, for $y\ge0$,
\[
\mathcal{J}_\nu(y):=\nu!\sum_{k=0}^\infty\frac{(-1)^ky^k}{k!\,(\nu+k)!},
\]
so that
\[
J_\nu(x)=\frac{(x/2)^\nu}{\nu!}\,\mathcal{J}_\nu\left((x/2)^2\right),
\]
and let
\[
\underline{\mathcal{J}}_\nu(y):=1-\frac{y}{\nu+1}+\frac{y^2}{2(\nu+1)(\nu+2)}
-\frac{y^3}{6(\nu+1)(\nu+2)(\nu+3)}.
\]
Then:
\begin{enumerate}
\item[\textup{(i)}] $\mathcal{J}_\nu(y)\ge\underline{\mathcal{J}}_\nu(y)$ for all $y\in[0,\pi^2]$;
\item[\textup{(ii)}] $\underline{\mathcal{J}}_\nu$ is strictly decreasing on $[0,\infty)$;
\item[\textup{(iii)}] $\underline{\mathcal{J}}_\nu(y)\ge1-\dfrac{y}{\nu+1}$ whenever $0\le y\le3(\nu+3)$.
\end{enumerate}
\end{lemma}

\begin{proof}
\begin{enumerate}
\item[\textup{(i)}] Set $t_k:=\dfrac{y^k\,\nu!}{k!\,(\nu+k)!}\ge0$, so
$\mathcal{J}_\nu(y)=\sum_k(-1)^kt_k$ converges absolutely.  The ratio
$\dfrac{t_{k+1}}{t_k}=\dfrac{y}{(k+1)(\nu+k+1)}$ satisfies, for $k\ge1$, $\nu\ge6$ and
$y\le\pi^2$,
\[
\frac{t_{k+1}}{t_k}\le\frac{\pi^2}{2(\nu+2)}\le\frac{\pi^2}{16}<1,
\]
so $t_1\ge t_2\ge t_3\ge\cdots$.  By absolute convergence we may regroup.
\[
\mathcal{J}_\nu(y)-\underline{\mathcal{J}}_\nu(y)=\sum_{k=4}^\infty (-1)^kt_k
=(t_4-t_5)+(t_6-t_7)+\cdots\ \ge0.
\]
\item[\textup{(ii)}] $\underline{\mathcal{J}}_\nu'(y)=-\dfrac{1}{\nu+1}\left(1-\dfrac{y}{\nu+2}
+\dfrac{y^2}{2(\nu+2)(\nu+3)}\right)$, and the quadratic factor in $y$ with
positive leading coefficient has the discriminant
\[
\frac1{(\nu+2)^2}-\frac{2}{(\nu+2)(\nu+3)}
=\frac{(\nu+3)-2(\nu+2)}{(\nu+2)^2(\nu+3)}
=\frac{-(\nu+1)}{(\nu+2)^2(\nu+3)}<0,
\]
hence is strictly positive. Thus $\underline{\mathcal{J}}_\nu'(y)<0$ for all $y$.
\item[\textup{(iii)}] The inequality is equivalent to
$\dfrac{y^2}{2(\nu+1)(\nu+2)}\ge\dfrac{y^3}{6(\nu+1)(\nu+2)(\nu+3)}$, i.e.\
$y\le3(\nu+3)$.
\end{enumerate}
\end{proof}

\begin{lemma}\label{lem:gbessel}
For $y>0$ let $g_y(s):=s\,J_1(ys)/y$ and $\gamma_y(s):=-s^{2}J_2(ys)/y$. Then
\begin{equation}\label{eq:gders}
\begin{split}
g_y'(s)&=sJ_0(ys),\qquad\qquad\qquad\;\,
g_y''(s)=J_0(ys)-ysJ_1(ys),\\
g_y'''(s)&=-yJ_1(ys)-y^{2}sJ_0(ys),\qquad
g_y''''(s)=-2y^{2}J_0(ys)+\frac ys\,J_1(ys)+y^{3}sJ_1(ys).
\end{split}
\end{equation}
and
\begin{equation}\label{eq:gammaders}
\gamma_y'(s)=-s^{2}J_1(ys),
\qquad
\gamma_y''(s)=-sJ_1(ys)-ys^{2}J_0(ys).
\end{equation}
In particular, if $J_1(y)=0$, then
\[
g_y(1)=0,\qquad g_y'(1)=g_y''(1)=J_1'(y)=J_0(y),\qquad g_y'''(1)=-y^{2}J_0(y),
\]
and
\[
\gamma_y(1)=\frac{J_0(y)}{y},\qquad \gamma_y'(1)=0.
\]
\end{lemma}

\begin{proof}
Since $s\,J_1(ys)=\frac1y\,YJ_1(Y)$ with $Y:=ys$, the identities $\left(YJ_1(Y)\right)'=YJ_0(Y)$ and $J_0'=-J_1$ \cite[\S10.6]{DLMF} give $g_y'(s)=sJ_0(ys)$ and then, differentiating repeatedly and using $J_1'(Y)=J_0(Y)-J_1(Y)/Y$, the remaining formulas in \eqref{eq:gders}. For \eqref{eq:gammaders}, the identity $J_2'(Y)=J_1(Y)-2J_2(Y)/Y$ \cite[\S10.6]{DLMF} gives $\gamma_y'(s)=-s^{2}J_1(ys)$, and a further differentiation together with $J_1'(Y)=J_0(Y)-J_1(Y)/Y$ gives $\gamma_y''$. The values at $s=1$ follow from $J_1(y)=0$, the last two using also the recurrence $J_2(Y)=2J_1(Y)/Y-J_0(Y)$.
\end{proof}

\begin{lemma}\label{lem:J2mono}
For $0<y<4$,
\[
J_3(y)\ \ge\ \frac{(y/2)^{3}}{6}\left(1-\frac{(y/2)^{2}}{4}\right)\ >\ 0,
\]
and consequently $y^{-2}J_2(y)$ is strictly decreasing on $(0,4)$.
\end{lemma}

\begin{proof}
In the series \eqref{eq:series} for $J_3$, the ratio of consecutive terms is $\frac{(y/2)^{2}}{(k+1)(k+4)}\le\frac{(y/2)^{2}}{4}<1$ for $y<4$, so the series is alternating with strictly decreasing terms and $J_3(y)$ is bounded below by its first two terms. The last claim follows from $\left(y^{-2}J_2(y)\right)'=-y^{-2}J_3(y)$ \cite[\S10.6]{DLMF}.
\end{proof}

\subsection*{Acknowledgements}
\phantomsection
\addcontentsline{toc}{section}{Acknowledgements}

\

JGS has been partially supported by NSF under Grants DMS-2245017, DMS-2247537, DMS-2434314, DMS-2554957, and by a Simons Fellowship.

ML has been partially supported by a Simons Foundation Award SFI-MPS-T-Institutes-00006117.

IP has been partially supported by NSERC and FRQNT.

BG, JGS, ML, and IP have been partially supported by the AI for Math Fund, an initiative by Renaissance Philanthropy with funding from XTX Markets.

The authors would like to thank the Isaac Newton Institute for Mathematical Sciences, Cambridge, for support and hospitality during the programme Geometric spectral theory and applications, where work on this paper was undertaken. This work was supported by EPSRC grant EP/Z000580/1.

\subsection*{Data availability statement}
\phantomsection
\addcontentsline{toc}{section}{Data availability statement}
The accompanying repository is available at
\begin{center}
\url{https://github.com/javigomez2001/BLP}.
\end{center}

\subsection*{AI usage disclosure}
\phantomsection
\addcontentsline{toc}{section}{AI usage disclosure}

The authors used a workflow combining ChatGPT Pro 5.5, Claude
Opus 4.8 and Gemini Pro 3.1 during the development of this work, including for suggesting proof
strategies and assisting with calculations. All mathematical statements, proofs, and verifications
were further developed and rigorously checked by the authors, who take full responsibility
for the results. The writeup of the results was done by the authors.

\phantomsection


\addcontentsline{toc}{section}{References}
\begin{thebibliography}{DSGSAH21}

{\small 

\bibitem[BLP09]{BLP09}
R.\ Benguria, M.\ Levitin, and L.\ Parnovski,
\emph{Fourier transform, null variety, and Laplacian's eigenvalues}.
J.\ Funct.\ Anal.\ 257 (2009), 2088--2123.
\mydoi{10.1016/j.jfa.2009.06.022}.

\bibitem[CLDP26]{CLDP26}
G. Cao-Labora and J. de Dios Pont,
\emph{Counterexamples to Schiffer's conjecture}.
Preprint (2026).
\myarXiv{2608.05114}.

\bibitem[CS26]{CS26}
M. J. Colbrook and G. Stepaniants,
\emph{A computer-assisted counterexample to the planar Pompeiu and Schiffer conjectures}.
Preprint (2026).
\myarXiv{2608.01579}.

\bibitem[DGSPA26]{DGP2026}
J.~Dahne, J.~G\'omez-Serrano, and J.~Pech-Alberich,
\emph{Monotonicity of the first Dirichlet eigenvalue of regular polygons}.
Preprint (2026).
\myarXiv{2601.16285}

\bibitem[DLMF]{DLMF}
NIST Digital Library of Mathematical Functions.
\url{https://dlmf.nist.gov/}, Release 1.2.2 of 2024-09-15.
F.~W.~J.~Olver, A.~B.~Olde~Daalhuis, D.~W.~Lozier, B.~I.~Schneider,
R.~F.~Boisvert, C.~W.~Clark, B.~R.~Miller, B.~V.~Saunders,
H.~S.~Cohl, and M.~A.~McClain, eds.

\bibitem[FLPS23]{Filonov-Levitin-Polterovich-Sher:polya-conjecture-balls}
N.~Filonov, M.~Levitin, I.~Polterovich, and D.~A.~Sher,
\emph{P{\'o}lya's conjecture for Euclidean balls}.
Invent. Math. \textbf{234}:1 (2023), 129--169.
\mydoi{10.1007/s00222-023-01198-1}.

\bibitem[FLPS26]{Filonov-Levitin-Polterovich-Sher:polya-conjecture-higher-dimensional-neumann-balls}
N.~Filonov, M.~Levitin, I.~Polterovich, and D.~A.~Sher,
\emph{P{\'o}lya's conjecture for higher-dimensional Neumann balls}.
Preprint (2026).
\myarXiv{2607.29305}.

\bibitem[GS19]{GomezSerrano:survey-cap-in-pde}
J.~G\'omez-Serrano,
\emph{Computer-assisted proofs in PDE: a survey}.
SeMA J. \textbf{76}:3 (2019), 459--484.
\mydoi{10.1007/s40324-019-00186-x}.

\bibitem[GHLZ26]{GHLZ2026}
C.~Gui, Y.~Hu, Q.~Li, and C.~Zhang,
\emph{Mixed torsion on right triangles and the P\'olya--Szeg\H{o} monotonicity
problem for regular polygons}.
Preprint (2026). 
\myarXiv{2606.13448}.


\bibitem[Jo17]{Johansson2019}
F.~Johansson,
\emph{Arb: Efficient Arbitrary-Precision Midpoint-Radius Interval Arithmetic}.
IEEE Transactions on Computers \textbf{66}:8 (2017), 1281--1292.
\mydoi{10.1109/TC.2017.2690633}.

\bibitem[Lo93]{Lorch}
L.~Lorch,
\emph{Some inequalities for the first positive zeros of Bessel functions}.
SIAM J.~Math.\ Anal.\ \textbf{24} (1993), 814--823.
\mydoi{10.1137/0524050}.

\bibitem[RF14]{RyadovkinFilonov2014}
K.~Ryadovkin and N.~Filonov,
\emph{The set of zeros of the Fourier transform of the characteristic function
of a symmetric convex domain}.
J. Math. Sciences \textbf{198} (2014), 747--760.
\mydoi{10.1007/s10958-014-1823-1}.

\bibitem[Tu11]{Tucker}
W.\ Tucker.
\emph{Validated Numerics: A Short Introduction to Rigorous Computations}.
Princeton University Press, 2011.
\mydoi{10.2307/j.ctvcm4g18}.

\bibitem[Wa44]{Watson}
G.\ N.\ Watson.
\emph{A Treatise on the Theory of Bessel Functions}.
2nd ed., Cambridge University Press, 1944.

\bibitem[Za84]{Zastavnyi}
V.~P. Zastavnyi,
\emph{Zero set of the Fourier transform of measures and the summation of double
Fourier series by methods of Bernshtein--Rogosinski type},
Ukr.\ Math.\ J.\ \textbf{36} (1984), 459--464.
\mydoi{10.1007/BF01086770}.

}
\end{thebibliography}
\end{document}